\documentclass[11pt, a4paper, fleqn]{article}
\usepackage[latin1,utf8]{inputenc}
\usepackage{amsmath}
\usepackage{amsthm}
\usepackage{amsfonts}
\usepackage{amssymb}
\usepackage[T1]{fontenc}
\usepackage[all]{xy}
\usepackage{bbm}
\usepackage{enumitem}
\usepackage{color}
\usepackage[hidelinks]{hyperref}
\usepackage{graphicx}
\usepackage{transparent}
\usepackage{graphics}
\usepackage{xcolor}
\usepackage{color}
\usepackage{geometry}
\usepackage{marginnote}
\usepackage{verbatim}
\usepackage{pgf,tikz}
\usepackage{mathrsfs}
\usetikzlibrary{arrows,decorations.pathmorphing,backgrounds,positioning,fit,petri,cd}

\title{Wall and Chamber Structure for finite-dimensional Algebras}
\date{}
\author{Thomas Br\"ustle, David Smith and Hipolito Treffinger}

\theoremstyle{plain} 
\newtheorem{theorem}{Theorem}[section]
\newtheorem{prop}[theorem]{Proposition}
\newtheorem{lem}[theorem]{Lemma}
\newtheorem{cor}[theorem]{Corollary}

\theoremstyle{remark}
\newtheorem{rmk}[theorem]{Remark}
\newtheorem{ex}[theorem]{Example}

\theoremstyle{definition}
\newtheorem{defi}[theorem]{Definition}
\newtheorem{conj}[theorem]{Conjecture}

\newcommand{\rep}[1]{%
  {%
    \tiny%
    \begin{matrix}%
      #1%
    \end{matrix}%
  }%
}

\newcommand{\gvec}[1]{%
  {%
    \tiny%
    \left(
    \begin{matrix}%
      #1%
    \end{matrix}%
    \right)
  }%
}

\def\Hom{\mbox{Hom}}

\def\coker{\mbox{coker}}
\def\im{\mbox{im}}
\def\top{\mbox{top}}
\def\Im{\mbox{\rm Im}}
\def\Fac{\mbox{Fac}\,}

\def\add{\mbox{\rm add}\,}
\def\mod{\mbox{\rm mod}\,}

\newcommand{\F}{\mathcal{F}}

\newcommand{\T}{\mathcal{T}}
\newcommand{\ra}{\rightarrow}

\newcommand{\D}{\mathfrak{D}}
\newcommand{\Ch}{\mathfrak{C}}
\newcommand{\Co}{\mathcal{C}}
\newcommand{\0}{\{0\}}
\newcommand{\C}{\mathcal{C}}
\newcommand{\N}{\mathbb{N}}
\newcommand{\R}{\mathbb{R}}

\newcommand{\rC}{\mathcal{C}\strut^{\mathrm{o}}}

\usepackage[nottoc]{tocbibind}

\begin{document}

\maketitle

\abstract{We use $\tau$-tilting theory to give a description of the wall and chamber structure of a finite dimensional algebra.
We also study $\D$-generic paths in the wall and chamber structure of an algebra $A$ and show that every maximal green sequence in $\mod A$ is induced by a $\D$-generic path.}

\section{Introduction}

Cluster algebras were introduced by Fomin and Zelevinski in \cite{FZ1}, prompting a lot of subsequent work on the subject. 
In particular, there are several representation theoretic categorifications of cluster algebras, see for instance \cite{CCS-ctiltedAn,BMRRT, GLS}.
More recently, Adachi, Iyama and Reiten introduced in \cite{AIR} $\tau$-tilting theory, an extension of the classical tilting theory that is compatible with the concept of mutation coming from cluster algebras. 
In doing so, $\tau$-tilting theory becomes a new categorification of cluster algebras with the novelty that its process of mutations can be applied to any finite-dimensional algebra, not only cluster-tilting algebras.
Therefore a number of concepts arising from cluster algebras can now be studied for any algebra from the $\tau$-tilting perspective.
\medskip

In \cite{FZ4} Fomin and Zelevinski introduced $g$-vectors, a set of vectors with integer entries that parametrize the cluster variables of a given cluster algebra. 
The first representation theoretic interpretation of these vectors was given by Dehy and Keller in \cite{DK}, and Adachi, Iyama and Reiten adapt them to $\tau$-tilting theory 
in \cite{AIR}.
Even if the name $g$-vector was new, these same vectors have been already considered in the representation theory of algebras before, in fact Auslander and Reiten studied them already in 1985 in \cite{AR1}.
Given an algebra $A$, the set of $g$-vectors in $\mod A$ enjoys many combinatorial properties.
For instance, it was proven in \cite{DIJ} that the $g$-vectors in $\mod A$ form a well behaved simplicial complex in $\mathbb{R}^n$, where $n$ is the number of non-isomorphic simple $A$-modules. 
Further properties of $g$-vectors have also been studied in different contexts, see for instance \cite{Rea,HPS}. 
\medskip

On the other hand, stability conditions were introduced in representation theory of quivers in seminal papers by Schofield \cite{Scho} and King \cite{Ki}.
Since then, the study of rings of quiver semi-invariants by Derksen and Weyman \cite{DW} has been expanded to the context of cluster algebras. 
The work of Igusa, Orr, Todorov and Weyman \cite{IOTW} shows that walls in the semi-invariant picture correspond to the $c$-vectors in cluster theory.
These vectors are also studied in quantum field theory, where they are interpreted as charges of BPS particles. 
It turns out that maximal green sequences, or more generally, maximal paths in the semi-invariant picture which are oriented in positive direction \cite{BHIT}, give rise to a complete sequence of charges, called spectrum of a BPS particle, see \cite[section 2]{BDP}. 
This phenomenon has already been observed by Seiberg and Witten in their study of $N=2$ SUSY with pure gauge group $SU(2)$ in \cite{SW}, which yields the wall and chamber structure of the Kronecker quiver. The moduli space of quantum field theories in this case has two fundamentally different regions, corresponding to the two possible maximal green paths in the semi-invariant picture.
\medskip

The semi-invariant picture of quiver representations has re-appeared in mathematical physics and mirror symmetry as scattering diagrams such as in Kontsevich and Soibelman's study of wall crossing in the context of Donaldson-Thomas invariants in integrable systems and mirror symmetry \cite{KS}.
Later, Gross, Hacking, Keel and Kontsevich studied in \cite{GHKK} the so-called cluster scattering diagram, proving several conjectures on cluster algebras. 
However, the cluster scattering diagram is an intrinsically geometric object. 
Therefore, taking an algebraic approach to the problem, Bridgeland introduces in \cite{B16} the algebraic scattering diagram and shows that both scattering diagrams are isomorphic if the algebra considered is hereditary.

In order to construct the scattering diagram of an algebra $A$, Bridgeland uses the partition of the real space $\mathbb{R}^n$ induced by the stability conditions over $\mod A$ introduced by King in \cite{Ki}. 
This partition of $\mathbb{R}^n$ is called the \textit{wall and chamber structure} of $A$.

The aim of this paper is therefore to join the concept of scattering diagrams and their wall and chamber structure as described in \cite{B16} with the combinatorial structure of the polyhedral fan associated with $\tau$-tilting modules as given in \cite{DIJ}, as well as to investigate maximal green sequences, and their continuous counterparts in the stability space.
\bigskip

\subsection{Content}

We recall  the notion of stability studied by King \cite{Ki}:
Let $A$ be an algebra whose Grothendieck group has rank $n$.
Then for any vector $\theta \in \R^n$, 
a non-zero module $M$ is called $\theta$-semistable  if its dimension vector $[M]$ is orthogonal to $\theta$, and $\langle\theta , [L] \rangle \le 0$  for every submodule $L$ of $M$. 
The stability space of an $A$-module $M$ is then defined as $$\D(M)=\{\theta\in \R^n : M \text{ is $\theta$-semistable}\}.$$

The stability space $\D(M)$ of $M$ is contained in the hyperplane orthogonal to $\theta$, but it could have smaller dimension. We say that $\D(M)$ is a \textit{wall} when $\D(M)$ has codimension one. 
Outside the walls, there are only vectors $\theta$ having no non-zero $\theta$-semistable modules.
Removing the closure of all walls we obtain a set
$$\mathfrak{R}=\R^n\setminus\overline{\bigcup\limits_{\substack{ M\in \mod A}}\D(M)}$$
whose  connected components $\mathfrak{C}$  are called  chambers. 
As connected components of an open set in $\R^n$, the chambers have dimension $n$.
This decomposition of $\R^n$ is called the wall and chamber structure of the algebra $A$ on $\R^n$.

The first aim of the paper is to study the category of $\theta$-semistable modules. 
It is known (see proposition \ref{wide}) that for fixed $\theta$ the category of $\theta$-semistable modules forms a wide subcategory of $\mod A$.
Using $\tau$-tilting theory, we are able to give a more precise statement (see Theorem \ref{stablemodcat}): 

\begin{theorem}
Let $(M,P)$ be a $\tau$-rigid pair and let $\alpha$ be a vector in the interior of the cone of $g$-vectors defined by $(M,P)$. 
Then the category of $\alpha$-semistable modules is equivalent to the module category of an algebra $C_{(M,P)}$. 
Moreover there are exactly $n-|M|-|P|$ nonisomorphic $\alpha$-stable modules, corresponding to the simple  $C_{(M,P)}$-modules.
\end{theorem}

Note that this was already shown using similar techniques for the $\mathbb{A}_n$ case in \cite{ITW}.

We further show in section 3 how the $\tau$-tilting fan introduced by Demonet, Iyama and Jasso in \cite{DIJ} can be embedded into King's stability manifold:
Each $\tau$-tilting pair $(M,P)$ yields a chamber $\Ch_{(M,P)}$, and one can give a complete description of the walls adjacent to the chamber $\Ch_{(M,P)}$:

\begin{theorem}\label{introwalls}
Let $A$ be a finite-dimensional algebra over an algebraically closed field. 
Then there is an injective function $\Ch$ mapping the $\tau$-tilting pair $(M,P)$ onto a chamber $\mathfrak{C}_{(M,P)}$ of the wall and chamber structure of $A$. 
Furthermore, if $A$ is $\tau$-tilting finite then $\Ch$ is also surjective.
\end{theorem}

We also define in section 3 a function $\T$ which assigns to each chamber $\Ch$ a torsion class $\T_{\Ch}$, and  we show that $\T_{\Ch_{(M,P)}}=\Fac M$.
\bigskip

Following \cite{B16},  we study in section 4 the $\D$-generic paths in the wall and chamber structure of an algebra $A$. These are  smooth paths crossing one wall at a time and such that the crossing is transversal, see definition \ref{defDgeneric}.
Moreover, given a $\D$-generic path $\gamma:[0,1] \to \mathbb{R}^n$ we associate to $\gamma(t)$ a torsion class for every $t\in[0,1]$.
This construction allows to show the following result.
 
 \begin{theorem}
Let $A$ be an algebra. 
Then every maximal green sequence is induced by a $\D$-generic path in the wall and chamber structure of $A$.

\end{theorem}

We finish section 4 with theorem \ref{Markoff} which provides a class of algebras not admitting a maximal green sequence.
These algebras are related to the cluster algebra of the one-punctured torus, and  have been object of intense studies in the context of cluster algebras, see for instance \cite[Example 35]{L-F-surfaces&potentials}, \cite{N-C-c&gvect} or \cite[Theorem 5.17]{DIJ}.
\bigskip

We refer to the textbooks \cite{ARS,AsSS,bookRalf} for background material. 

\subsection{Acknowledgements}

This paper is a revised version of one part of a preprint \cite{BST}. The authors would like to thank Kiyoshi Igusa, Patrick Le Meur, as well as an anonymous referee for for their input and suggestions.

When we were working on this article, we learned that Yurikusa is using the $g$-vectors of $2$-term silting complexes of $D^b(A)$ to describe a bijection between left wide subcategories of $\mod A$ and the left finite semistable subcategories of $\mod A$, induced by a linear map $\theta$ in a similar way as we do in this article, see \cite{Yur}.
\bigskip

The first and the second author were supported by Bishop's University, Université de Sherbrooke and NSERC of Canada. The third author was supported by
the EPSRC funded project EP/P016294/1.

\section{Preliminaries}\label{prelim}

We consider a finite dimensional algebra $A$ over a field $k$ of the form $A = kQ/I$ where $Q$ is a quiver and $I$ an admissible ideal of the path algebra $kQ$. 
For an algebraically closed field $k$, every finite-dimensional $k$-algebra is Morita-equivalent to an algebra of the form $kQ/I$.
We study the category $\mod A$ of finitely generated modules  over $A$. 

Its Grothendieck group $K_0(A)$ is free abelian of finite rank $n$, where $n$ is the number of vertices of the quiver $Q$. 
In this paper we consider the isomorphism that assigns to the class $[M]$ in $K_0(A)$ of any $A$-module its dimension vector. 
By abuse if notation $[M]$ represents both the class of $M$ in $K_0(A)$ and its dimension vector. 

The $\tau$-tilting theory was introduced by Adachi, Iyama and Reiten in \cite{AIR}, where $\tau$ denotes the Auslander-Reiten translation in $\mod A$. 
It extends classical tilting theory from the viewpoint of mutation, providing a framework for studying problems arising from cluster algebras. 
In this paper the $\tau$-rigid and $\tau$-tilting pairs play a central role. 
They are defined as follows. 

\begin{defi}\cite[Definition 0.1 and 0.3]{AIR}\label{list}
Consider an $A$-module $M$ and a projective $A-$module $P$. The pair $(M,P)$ is said \textit{$\tau$-rigid} if:
		\begin{itemize}
			\item $\Hom_{A}(M,\tau M)=0$;
			\item $\Hom_{A}(P,M)=0$.
		\end{itemize}
The notion of a $\tau$-rigid pair generalizes the one of a $\tau$-rigid $A$-module $M$, which is given just by the first condition $\Hom_{A}(M,\tau M)=0$.
We say moreover that a $\tau$-rigid pair $(M,P)$ is \textit{$\tau$-tilting} if $|M|+|P|=n$, and \textit{almost $\tau$-tilting} if $|M|+|P|=n-1$.
Here we denote by $|X|$ the number of  direct summands of $X$.        
\end{defi}

Note that we assume our algebras and modules to be basic, that is, all its indecomposable direct summands are non-isomorphic.	
That is, we write a $\tau$-rigid pair $(M,P)$ as $M=\bigoplus_{i=1}^kM_i$ and $P=\bigoplus_{j=k+1}^tP_j$ 
with all $M_i$ indecomposable and non-isomorphic, and all $P_j$ indecomposable projective and non-isomorphic, so 
$(M,P)$ is $\tau$-tilting precisely when $t=n$.

As usual, we denote the right perpendicular category of a module $M$ by 
$$M^{\perp}=\{X\in\mod A: \Hom_A(M,X)=0 \}$$ and, dually, 
$$^{\perp}M=\{Y\in\mod A: \Hom_A(Y,M)=0\}.$$
Recall that, for a given $A$-module $M$, the full subcategory Fac$M$ of $\mod A$ is defined as
$$\mbox{Fac}M = \{ X \in \mod A : \mbox{ there is an epimorphism } M^l \to X \mbox{ for some } l \in \N\}.$$
Remember that a torsion pair $(\T,\F)$ is a pair of full subcategories of mod $A$ such that :
\begin{itemize}
\item $\Hom_A (X, Y)=0$ for every $X\in \T$ and every $Y\in \F$.
\item Maximality of $\T$: If $X$ is an $A$-module with $\Hom_A(X,F) = 0$ for all $F$ in $\F$, then $X$ belongs to $\T$.
\item Maximality of $\F$: If $Y$ is an $A$-module with $\Hom_A(T,Y) = 0$ for all $T$ in $\T$, then $Y$ belongs to $\F$.
\end{itemize}
Note that a torsion pair automatically satisfies that
\begin{itemize}
\item $\T$ is closed under quotients and extensions,
\item $\F$ is closed under submodules and extensions.
\end{itemize}

We now describe how $\tau$-tilting theory allows to describe all functorially finite torsion classes of $\mod A$ in terms of $\tau$-tilting pairs.

\begin{theorem}\cite[Theorem 2.7]{AIR}\cite[Theorem 5.10]{ASfuncfiniteFac}\label{TorClass}

Defining $\Phi(M,P)=\emph{Fac} M$ yields a function from $\tau$-rigid pairs to functorially finite torsion classes.
Moreover, $\Phi$ is a bijection when restricted to $\tau$-tilting pairs.
\end{theorem}

Every torsion pair $(\T, \F)$ in $\mod A$ has the following property. 
For each $A$-module $N$ there exists a short exact sequence, referred to as the canonical short exact sequence for $N$, 
$$0\ra tN\ra N\ra N/tN\ra 0$$
with $tN\in\T$ and $N/tN\in\F$. 
The module $tN$ is unique up to isomorphism, and is called the {\em trace} of $N$ in $\T$. 
For a $\tau$-rigid pair, the trace of $N$ can be obtained as follows:

\begin{lem}\label{approx}
Let $(M,P)$ be a $\tau$-rigid pair and $N$ an $A-$module. 
Then the trace $tN$ of $N$ with respect to the torsion pair $(\text{Fac } M, M^\perp)$ can be computed as $tN = \Im f$ where $f: M'\to N$ is the minimal right $\add M$-approximation of $N$.
\end{lem}

\begin{proof}
A  morphism $f:M' \to N$ is called a right $\add M$-approximation of $N$ if $M' \in \add M$ and  for all $X \in \add M$ the induced map $\Hom(X,f)$ is surjective.
We have that $tN\in \Fac M$ by definition, therefore there is a natural number $l$ and an epimorphism $p:M^l \to tN$. 
This $p$ factors through $f$ because $f$ is the minimal right $\add M$-approximation of $N$. 
Hence $tN$ is isomorphic to a submodule of Im$f$.
Conversely, Im $f$ is a submodule of $N$ which belongs to $\Fac M$, and therefore Im $f$ is isomorphic to a submodule of $tN$. 
\end{proof}

Another important feature of $\tau$-tilting pairs is the fact that every almost $\tau$-tilting pair can be completed to a $\tau$-tilting pair in exactly two different ways:

\begin{theorem}\label{mutation}\cite[Theorem 2.8]{AIR}
Let $(M,P)$ be an almost $\tau$-tilting pair. Then there exist exactly two different $\tau$-tilting pairs $(M_1, P_1)$ and $(M_2, P_2)$ such that $M$ is a direct summand of both $M_1$ and $M_2$, and $P$ is a direct summand of $P_1$ and $P_2$. 
In that case we say that $(M_1, P_1)$ and $(M_2, P_2)$ are a \textit{mutation} of each other. 
\end{theorem}

We further recall from \cite{AIR} that in the setting of the previous theorem, one torsion class covers the other, say $\text{Fac} M=\text{Fac} M_1\subsetneq\text{Fac}M_2$.   
Writing $M_2=M'\oplus M$ with an indecomposable $\tau$-rigid module $M'$, we say that $M_2$ is a left mutation of $M_1$.

More generally, one can consider the problem of finding all $\tau$-tilting pairs having a given $\tau$-rigid pair $(M,P)$ as a direct summand. 
This problem was solved by Jasso in \cite{J} using a procedure that he called \textit{$\tau$-tilting reduction}. 
Here we give a brief summary of that process. 

First, by Theorem \ref{TorClass} one knows that $(M,P)$ yields the torsion class $\Fac M$. 
But there is another torsion class given by $(M,P)$, namely the class $^\perp(\tau M) \cap P^\perp$. 
By \cite[Theorem 2.12]{AIR}, these two torsion classes coincide if and only if $(M,P)$ is a $\tau$-tilting pair. 
Moreover, Theorem \ref{TorClass} together with \cite[Theorem 2.9]{AIR} implies the existence of a $\tau$-tilting pair of the form $(M\oplus M', P)$ such that $\Fac (M\oplus M')= ^\perp(\tau M)\cap P^\perp$.

In the endomorphism algebra $B_{(M,P)}=\text{End}_A(M\oplus M')$, there is an idempotent element $e_{(M,P)}$ associated to the $B_{(M,P)}$-projective module $\Hom_{A}(M\oplus M', M)$. 
We define the algebra $C_{(M,P)}$ as the quotient of $B_{(M,P)}$ by the ideal generated by $e_{(M,P)}$, \textit{i.e.},
$$C_{(M,P)}:=B_{(M,P)}/B_{(M,P)} e_{(M,P)} B_{(M,P)}.$$

Now we are able to state one of the main theorems of \cite{J}.

\begin{theorem}\cite[Theorem 3.8]{J}\label{tautiltred}
Let $(M,P)$ be a $\tau$-rigid pair in $\emph{mod} A$. 
Then the functor $$\emph{Hom}_A (M\oplus M', -): \emph{mod} A \to \emph{mod} B_{(M,P)}$$ induces an equivalence of categories 
$$F:M^\perp \cap\; ^\perp(\tau M) \cap P^\perp \to \emph{mod} C_{(M,P)}$$ between the perpendicular category $M^\perp \cap \;^\perp(\tau M) \cap P^\perp$ of $(M,P)$ and the module category $\emph{mod} C_{(M,P)}$.
\end{theorem}

\section{Hyperplane arrangements and cone complex}\label{sect:conecomplex}

We first describe in this section the wall and chamber structure of $\R^n$ induced by the algebra $A$, following freely the exposition in \cite{B16}. 
Traditionally, stability conditions are formulated with respect to a linear form 
$K_0(A)\otimes\mathbb{R} \cong \R^n \ra \mathbb{R}$. However, we prefer to draw a stability condition $\theta$ and the class $[M] \in K_0(A)$ of a module $M \in \mod A$ in the same picture, therefore we work with vectors $\theta \in \R^n$.
To comply with the usual notation, we sometimes write $\theta(M)$ for the standard inner product $\langle \theta , [M] \rangle$  on $\R^n$. 
\subsection{The wall and chamber structure of an algebra}

We recall from King in \cite{Ki} the notion of stability for modules:

\begin{defi}\cite[Definition 1.1]{Ki}\label{stable}
For $\theta \in \R^n$, 
a non-zero module $M\in\mod A$ is called \textit{$\theta$-stable}  if it is orthogonal to $\theta$, that is, $\theta(M)=0$, and $\theta(L)<0$  for every proper submodule $L$ of $M$. 
Moreover, a module $M$ orthogonal to $\theta$ is called \textit{$\theta$-semistable} if
$\theta(L)\leq 0$ for every submodule $L$ of $M$.
\end{defi}

A central notion in this section is given by the set of all values $\theta$ that turn a given module semistable:

\begin{defi}
The \textit{stability space of an $A$-module $M$} is $$\D(M)=\{\theta\in \R^n : M \text{ is $\theta$-semistable}\}.$$
\end{defi}

It is clear from the above definitions that $\D(M)$ is a cone given by intersections of hyperplanes in $\R^n$. 
We say the stability space $\D(M)$ of $M$ is  a \textit{wall} when $\D(M)$ has codimension one. 
We refer in this case to $\D(M)$ as the wall defined by $M$.
Not every $\theta$ belongs to the stability space $\D(M)$ for some nonzero module $M$,  for instance the vector
$\theta=(1, \ldots, 1)$ is never orthogonal to the dimension vector of a non-zero module.
More generally, none of the vectors having all strictly positive or all strictly negative entries is orthogonal to any dimension vector.
This leads to the following definition.

\begin{defi}
Let $$\mathfrak{R}=\R^n\setminus\overline{\bigcup\limits_{\substack{ M\in \mod A}}\D(M)}$$
denote the maximal open set of $\theta$ having no $\theta$-semistable non-zero modules.
Then a connected component $\mathfrak{C}$ 
of $\mathfrak{R}$ is called a \textit{chamber}. 
\end{defi}

We illustrate this wall and chamber structure by the following example.

\begin{ex}\label{runningexample}
Consider the path algebra $\mathbb{A}_2=kQ$ of the quiver $Q=\xymatrix{1\ar[r]& 2}$. Its  Auslander-Reiten quiver is as follows:

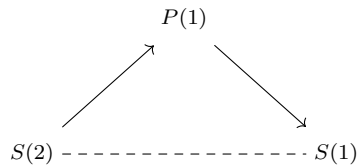
\begin{center}
			\begin{tikzpicture}[line cap=round,line join=round ,x=2.0cm,y=1.8cm]
				\clip(-1.2,-0.1) rectangle (1.2,1.1);
					\draw [->] (-0.8,0.2) -- (-0.2,0.8);
					\draw [dashed] (-0.8,0.0) -- (0.8,0.0);
					\draw [->] (0.2,0.8) -- (0.8,0.2);
				
				\begin{scriptsize}
					\draw[color=black] (-1,0) node {$S(2)$};
					\draw[color=black] (0,1) node {$P(1)$};
					\draw[color=black] (1,0) node {$S(1)$};
				\end{scriptsize}
			\end{tikzpicture}
		\end{center}

The wall and chamber structure of $\mathbb{A}_2$ is illustrated in Figure \ref{w&cA2}.

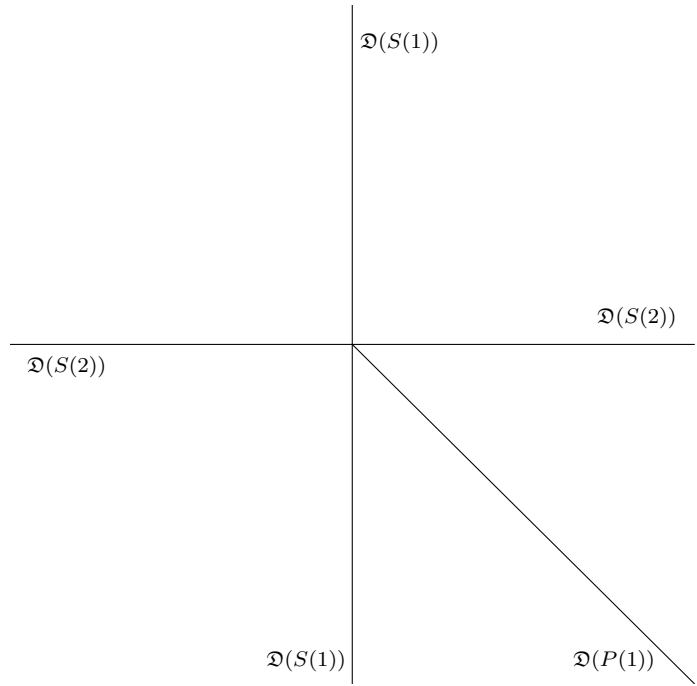
\begin{figure}
			\begin{center}
				\begin{tikzpicture}[line cap=round,line join=round,>=triangle 45,x=3.0cm,y=3.0cm]
					\clip(-1.5,-1.5) rectangle (1.5,1.5);
						\draw (0.,0.) -- (0.,1.5);
						\draw (0.,0.) -- (0.,-1.5);
						\draw [domain=-1.5:0.0] plot(\x,{(-0.-0.*\x)/-1.});
						\draw [domain=0.0:1.5] plot(\x,{(-0.-1.*\x)/1.});
						\draw [domain=0.0:1.5] plot(\x,{(-0.-0.*\x)/1.});
					\begin{scriptsize}
						\draw[color=black] (0,1.25) node[anchor= south west] {$\mathfrak{D}(S(1))$};
						\draw[color=black] (-0.2,-1.4) node {$\mathfrak{D}(S(1))$};
						\draw[color=black] (-1.25,-0.1) node {$\mathfrak{D}(S(2))$};
						\draw[color=black] (1.15,-1.4) node {$\mathfrak{D}(P(1))$};
						\draw[color=black] (1.25,0.12) node {$\mathfrak{D}(S(2))$};
					\end{scriptsize}
				\end{tikzpicture}
			\end{center}
			\caption{Wall and chamber structure for $\mathbb{A}_2$}
            \label{w&cA2}
		\end{figure}
\end{ex}

Note that, in general there may be infinitely many walls that can even form dense regions in $\R^n$. 
We refer to the examples in \cite{Rea}, the example 1.3 in \cite{B16}, and the figures in \cite{DW}, the latter being formulated in the language of Schur roots. 
It is observed in \cite{Ki,Ru} that the dimension vector of a $\theta$-stable module is a brick in $\mod A$ for every $\theta \in \mathbb{R}^n$.

\subsection{The simplicial complex of $\tau$-rigid pairs}

We describe in this subsection the simplicial complex defined by Demonet, Iyama and Jasso in \cite{DIJ}. 
We start by recalling the notion of $g$-vectors first introduced by Dehy and Keller in \cite{DK} and later adapted to module categories by Adachi, Iyama and Reiten in \cite{AIR}. 

\begin{defi}
Let $M$ be an $A$-module. 
Choose the minimal projective presentation $$P_1\longrightarrow P_0\longrightarrow M\longrightarrow 0$$  of $M$, where $P_0=\bigoplus\limits_{i=1}^n P(i)^{c_i}$ and $P_1=\bigoplus\limits_{i=1}^n P(i)^{c'_i}$. 
Then the $g$-vector of $M$ is defined as
$$g^M=(c_1-c'_1, c_2-c'_2,\dots, c_n-c'_n).$$
The $g$-vector of a $\tau$-rigid pair $(M,P)$ is defined as $ g^{M}-g^{P}$.
\end{defi}

\begin{rmk}
Note that the canonical basis of $\mathbb{Z}^n$ correspond to
$$\{g^{P(1)}, \dots, g^{P(n)}\}$$
where $A=P(1)\oplus \dots \oplus P(n)$ is the decomposition of $A$ as a sum of indecomposable projective $A$-modules.
\end{rmk}

The following two results give important properties of the $g$-vectors of $\tau$-rigid and $\tau$-tilting pairs. 
The first of them  extends the previous remark to all $\tau$-tilting pairs.

\begin{theorem}\cite[Theorem 5.1]{AIR}\label{LI}
Let $(M,P)$ be a $\tau$-tilting pair, and denote the indecomposable summands as $M=\bigoplus_{i=1}^kM_i$ and $P=\bigoplus_{j=k+1}^n P_j$. Then the set
$$\{g^{M_1},\dots, g^{M_k}, -g^{P_{k+1}}, \dots, -g^{P_n}\}$$
forms a basis for $\mathbb{Z}^n$.
\end{theorem}

\begin{theorem}\cite[Theorem 5.5]{AIR}\label{inj-g-vecteurs}
Let $M$ and $N$ be two $\tau$-rigid $A$-modules. Then $g^M=g^N$ if and only if $M$ is isomorphic to $N$.
\end{theorem}
 
The following result of Auslander and Reiten \cite{AR1} yields a homological interpretation for the inner product of the $g$-vector of a module $M$ with the dimension vector $[N]$ of a module $N$.
For the sake of simplicity, we abbreviate $\dim_k(\Hom_A(M,N))$ by $\emph{hom}_A(M,N)$. 

\begin{theorem}[\cite{AR1}, Theorem 1.4.(a)]\label{formula}
Let $M$ and $N$ be modules over the algebra $A$. Then we have
$$\langle g^{M},[N]\rangle=hom_A(M,N)-hom_A(N,\tau_A M)$$
\end{theorem}

The following formula is an application of \cite[Theorem 1.4]{AR1} to $\tau$-rigid pairs. It plays a central role in the remaining part of the paper.

\begin{cor}\label{tauformula}
Let $(M,P)$ be a $\tau$-rigid pair and $N$ be an $A$-module. Then
$$\langle g^{M}-g^{P},[N]\rangle=hom_A(M,N)-hom_A(N,\tau_A M)-hom_A(P,N).$$
\end{cor}

\begin{proof}
It suffices to apply Theorem \ref{formula} twice and notice that $\tau P=0$.
\end{proof}

\begin{rmk}
Theorem \ref{formula} shows that for a $\tau$-tilting module $M$ of projective dimension at most one the inner product $\langle g^{M},[N]\rangle$ coincides with the Euler-Ringel form evaluated at the classes $[M],[N]$, which is also the skew-symmetric form considered by Bridgeland in \cite{B16}. 
\end{rmk}

In the next subsection we will show how the abstract simplicial complex of $\tau$-tilting pairs introduced in \cite{DIJ} is closely related to the support of the scattering diagrams introduced in \cite{B16}. 
In order to do that, we need to introduce a piece of notation.

For a set of linearly independent vectors $L=\{v_1, \dots, v_t\}$ in $\R^n$, we consider the \textit{polyhedral cone $\mathcal{C}_L$} of $L$ as $$\mathcal{C}_L=\left\{\sum\limits_{i=1}\limits^{t} \alpha_iv_i : \alpha_i \ge 0 \text{ for every $1\leq i\leq t$}\right\}.$$ 
Its interior is 
$$\rC_L=\left\{\sum\limits_{i=1}\limits^{t} \alpha_iv_i : \alpha_i>0 \text{ for every $1\leq i\leq t$}\right\}.$$
Since $L$ is a linearly independent set, the coefficients $\alpha_i$ of a vector $v = \sum\limits_{i=1}\limits^{t} \alpha_iv_i $ in $\Co_L$ are uniquely determined by $v$, and in particular all of them are non-negative. 
Consider now  a $\tau$-rigid pair $(M,P)$ where $M=\bigoplus\limits_{i=1}^kM_i$ and $P=\bigoplus\limits_{j=k+1}\limits^t P_j$ are the decomposition of $M$ and $P$ as sums of indecomposable modules, respectively. By Theorem \ref{LI}, their $g$-vectors are linearly independent, and we consider the polyhedral cone $\mathcal{C}_{(M,P)}$ given by the set of vectors 
$L = \{g^{M_1},\dots, g^{M_k}, -g^{P_{k+1}}, \dots, -g^{P_t}\}$:
$$\mathcal{C}_{(M,P)}=\left\{\sum\limits_{i=1}\limits^{k} \alpha_ig^{M_i}-\sum\limits_{j=k+1}\limits^{t} \alpha_j g^{P_j} : \alpha_i \ge 0 \text{ for every $1\leq i\leq t$}\right\}.$$ 
The {\em polyhedral fan} of $A$ studied in \cite{DIJ} is the collection of cones $\mathcal{C}_{(M,P)}$. The $\tau$-tilting pairs yield the $n$-dimensional cones, which are separated by facets (cones of codimension one), which are given by the almost $\tau$-tilting pairs.
For any $\tau$-rigid pair $(M,P)$, the vector  $\alpha(M,P) = \sum\limits_{i=1}\limits^{k} \alpha_ig^{M_i}-\sum\limits_{j=k+1}\limits^{t} \alpha_j g^{P_j}$ in $\Co_{(M,P)}$ has uniquely determined non-negative coefficients $\alpha_i, \alpha_j$.
This yields a linear form on $\R^n$ defined by

$$\langle\alpha(M,P),[N])\rangle =\sum^{k}_{i=1}\alpha_i\emph{hom}(M_i,N)-\sum_{i=1}^k\alpha_i\emph{hom}(N,\tau M_i)-\sum_{j=k+1}^t\alpha_j\emph{hom}(P_j,N).$$

\subsection{From $\tau$-tilting pairs to chambers}\label{chamber-results}
We show in this subsection how the polyhedral cone defined by $g$-vectors can be embedded into the wall and chamber structure of a given algebra. 

\begin{lem}\label{lemtrace}
Let $(M,P)$ be a $\tau$-rigid pair and let $\alpha(M,P)\in\mathcal{C}_{(M,P)}$ be in the cone defined by $(M,P)$. 
Then for every $N\in \emph{mod} A$, the trace $tN$ of $ N$ in $\emph{Fac} M$ satisfies
$$\langle\alpha(M,P),[tN])\rangle\geq 0.$$
Moreover the inequality is strict if $tN \neq 0 $ and $\alpha(M,P)$ lies in the interior $\rC_{(M,P)}$.
\end{lem}

\begin{proof}

Consider the  short exact sequence 
$$0\rightarrow tN\rightarrow N\rightarrow N/tN\rightarrow 0$$
with $tN\in \Fac M$ and $N/tN\in M^{\perp}$.
Assuming there exists a non-zero morphism $f: tN \to \tau M$ from $tN$ to $\tau M$, we could compose $f$ with an epimorphism $p: M^l \to tN$ from some $M^l$ to $tN$ and create a non-zero element $g \in \Hom_A(M,\tau M)$, contradicting the assumption that $M$ is $\tau$-rigid. 
Consequently $hom(tN, \tau M)=0$.

Likewise, any morphism $f:P\ra tN$ factors by projectivity of $P$ through some  epimorphism $p: M^l\ra tN$. 
Therefore $f=0$ because $\Hom_A(P,M)=0$ and thus $\emph{hom}_A(P,tN)= 0$. 
This yields
$$\langle\alpha(M,P),[tN])\rangle =\sum^{k}_{i=1}\alpha_i\emph{hom}(M_i,tN)-\sum_{i=1}^k\alpha_i\emph{hom}(tN,\tau M_i)-\sum_{j=k+1}^t\alpha_j\emph{hom}(P_j,tN)$$
$$=\sum^{k}_{i=1}\alpha_i\emph{hom}(M_i,tN) \; \geq \; 0\; .$$
Moreover, $\alpha(M,P)$ lies in the interior $\rC_{(M,P)}$ precisely when all $\alpha_i >0$, and 
$tN \neq 0 $ means that some $\emph{hom}(M_i,tN)$ is non-zero, since $tN$ belongs to $\Fac M$. Thus the inequality is strict in this case.
\end{proof}

\begin{prop}\label{catsemistables}
Let $(M,P)$ be a $\tau$-rigid pair and let $\alpha(M,P)\in\rC_{(M,P)}$. 
Then $N$ is an $\alpha(M,P)$-semistable module if and only if $N\in M^{\perp} \cap {^\perp}(\tau M) \cap P^{\perp}$.
\end{prop}

\begin{proof}
Suppose that $N\in M^{\perp} \cap {^\perp}(\tau M) \cap P^{\perp}$. 
Then $$\emph{hom}(M,N)=\emph{hom}(N,\tau_A M)=\emph{hom}(P,N)=0.$$ 
Therefore $$\langle\alpha(M,P),[N]\rangle=0$$ by Corollary \ref{tauformula}. 
Moreover, for every submodule $L$ of $N$ we have that $\emph{hom}(M,L)=0$ because $\emph{hom}(M,N)=0$. 
Likewise $\emph{hom}(P,L)=0$. 
Then $$\langle\alpha(M,P),[L]\rangle \leq -\sum_{i=1}^k\alpha_i\emph{hom}(L,\tau_A M_i)\leq 0.$$ 
Hence $N$ is an $\alpha(M,P)$-semistable module. 

Conversely, suppose that $N$ is an $\alpha(M,P)$-semistable module, thus $$\langle\alpha(M,P),[L]\rangle \leq 0$$
for every submodule $L$ of $N$. 
In particular $$\langle\alpha(M,P),[tN]\rangle \leq 0$$ where $tN$ is the trace of $N$ in $\Fac M$. 
But lemma \ref{lemtrace} implies $$\langle\alpha(M,P),[tN]\rangle > 0$$ when $\alpha(M,P)\in\rC_{(M,P)}$ and $tN$ is non-zero, thus we conclude  $tN=0$. 
This yields $N \cong N/tN \in M^{\perp}$.
Moreover, from 
$$0 = \langle\alpha(M,P),[N]\rangle =  \sum_{i=1}^k\alpha_i\emph{hom}(M_i,N)-\sum_{i=1}^k\alpha_i\emph{hom}(N,\tau M_i)-\sum_{j=k+1}^t\alpha_j\emph{hom}(P_j,N)$$
we conclude $\emph{hom}(N,\tau M)=\emph{hom}(P,N)=0$ since $\emph{hom}(M,N) = 0$ and all $\alpha_i,\alpha_j >0$. Therefore $N\in  M^{\perp} \cap {^\perp}(\tau M) \cap P^{\perp}$. 
\end{proof}

The previous result provides a geometric proof for the fact that $$M^{\perp} \cap {^\perp}(\tau M) \cap P^{\perp} = 0$$ when $(M,P)$ is a $\tau$-tilting pair: 
We know in this case that the spanning vectors of the cone $\Co_{(M,P)}$ form a basis for $\R^n$, and thus only the zero vector can be orthogonal to every vector $\alpha(M,P)\in\Co_{(M,P)}$.
In other words, there is no $\alpha(M,P)$-stable module. 
In general, it is shown in \cite{J} that the full subcategory $M^{\perp} \cap {^\perp}(\tau M) \cap P^{\perp}$ of $\mod A$ is equivalent to the category of modules $\mod C_{(M,P)}$ where the algebra $C_{(M,P)}$ is obtained as endomorphism ring of a module $M'$ over $A$ such that $(M \oplus M', P)$  is $\tau$-tilting, see Theorem \ref{tautiltred}.
This allows us to give in the following result the precise number of $\alpha(M,P)$-stable modules when $(M,P)$ is a $\tau$-rigid pair:

\begin{theorem}\label{stablemodcat}
Let $(M,P)$ be a $\tau$-rigid pair and let $\alpha(M,P)$ be in the interior $\rC_{(M,P)}$ of the cone associated with $(M,P)$. 
Then the category of $\alpha(M,P)$-semistable modules is equivalent to the module category of the algebra $C_{(M,P)}$. 
Moreover there are exactly $rk(K_0(A))-|M|-|P|$ nonisomorphic $\alpha(M,P)$-stable modules.
\end{theorem}

\begin{proof}
The first part of the statement follows directly from proposition \ref{catsemistables} and Theorem \ref{tautiltred}. 
Note that $rk(K_0(C_{(M,P)}))=rk(K_0(A))-|M|-|P|$ by construction of $C_{(M,P)}$ as explained in section \ref{prelim}. 
Moreover, the functor $F$ of proposition \ref{tautiltred} induces a bijection between the isomorphism classes of $\alpha(M,P)$-stable modules and the isomorphism classes of simple $C_{(M,P)}$-modules.
Therefore the number of nonisomorphic $\alpha(M,P)$-stable modules coincides with the rank of $K_0(C_{(M,P)})$ the Grothendieck group of $\mod C_{(M,P)}$, finishing the proof.
\end{proof}

We can now use Theorem \ref{stablemodcat}  to relate $\tau$-tilting pairs to chambers:

\begin{prop}\label{chambers}
Let $A$ be a finite-dimensional $k$-algebra. 
Then the interior $\rC_{(M,P)}$ of the positive cone  associated with a $\tau$-tilting pair $(M,P)$ defines a chamber in the wall and chamber structure of $A$.
\end{prop}

\begin{proof}
Let $(M,P)$ be a $\tau$-tilting pair and let $\alpha(M,P)\in\rC _{(M,P)}$. 
As we discussed before Theorem \ref{stablemodcat}, this implies that the category of $\alpha(M,P)$-semistable modules consist only of the zero object. 
Therefore $\alpha(M,P)$ belongs to a chamber $\mathfrak{C}$. 
Moreover, every vector in $\rC_{(M,P)}$ belongs to the same chamber $\mathfrak{C}$ because $\rC_{(M,P)}$ is connected. 
Thus $\rC_{(M,P)} \subset \mathfrak{C}$. 

For every non-zero vector $\beta(M,P)$ in the boundary $$\partial\Co_{(M,P)} = \Co_{(M,P)} \backslash \rC_{(M,P)} $$ of $\Co_{(M,P)}$ there exist indices $1\leq i\leq n$ such that $\beta_i=0$. 
That is, $\beta(M,P)$ lies in a smaller-dimensional cone defined by a $\tau$-rigid pair $(M',P')$ obtained from $(M,P)$ by removing summands corresponding to these indices. 
If every entry of $\beta(M,P)$ is zero, the category of $\beta(M,P)$-semistable modules coincides with $\mod A$.
Otherwise, some $\beta_j$ is non-zero.
In that case, the vector $\beta(M,P)$ lies in the interior of a suitable cone $\rC_{(M',P')}$.
Hence, Theorem \ref{stablemodcat} yields the existence of a $\beta(M,P)$-stable module.
This implies $\rC_{(M,P)} = \mathfrak{C}$. 
\end{proof} 

In view of the previous result, the interior of the positive cone $\Co_{(M,P)}$ of a $\tau$-tilting pair $(M,P)$ will be referred to as the chamber induced by $(M,P)$ and denoted by $\Ch_{(M,P)}:= \rC_{(M,P)}$.

Another immediate consequence of theorem \ref{stablemodcat} is the following.

\begin{cor}\label{coralmost}
Let $(M,P)$ be an almost $\tau$-tilting pair. 
Then $\Co_{(M,P)}$ has codimension 1 and is included in a wall of the wall and chamber structure of \emph{mod}$A$.
\end{cor}
\begin{proof} 
The cone $\Co_{(M,P)}$ of an almost $\tau$-tilting pair is formed by $n-1$ linearly independent vectors, thus it has codimension 1. 
Theorem \ref{stablemodcat} guarantees that every vector $\alpha(M,P)$ in the interior of $\Co_{(M,P)}$ admits an $\alpha(M,P)$-stable module, thus intersects no chamber and must be contained in a wall.
\end{proof}

We would like to point out that the facets $\Co_{(M,P)}$ of the polyhedral fan need not be walls in the wall and chamber structure: 
In general one wall $\D(N)$ can be made of more than one facet, see for example the wall $\D(S(2))$ in Figure \ref{tauw&cA_2}.
While the corollary above shows that the positive cones of almost $\tau$-tilting pairs are included in a wall $\D(N)$, it gives little information about the module $N$ defining that wall. 
A method to construct this module explicitly follows from the theory developed in \cite{AIR}:

\begin{prop}\label{walls}
Let $(M,P)$ be an almost $\tau$-tilting pair. 
Then $\Co_{(M,P)}$ is included in the wall $\D(N)$, where $N$ is constructed as follows: 
Let $(M_1,P_1)$ and $(M_2,P_2)$ be the two $\tau$-tilting pairs containing $M$ and $P$ as a direct factor. Order them such that $\emph{Fac} M_1 \subset \emph{Fac} M_2$, and write $M_2=M'\oplus M$ for some $\tau$-rigid module $M'$.
Then $N$ is the cokernel of the right \emph{add}$M$-approximation of $M'$. 
\end{prop}

\begin{proof}
Recall from section \ref{prelim} that  there exists an indecomposable $\tau$-rigid module $M'$ such that $M_2=M'\oplus M$. 
By Lemma \ref{approx}, the short exact sequence induced by the right add $M$-approximation is as follows, where $tM'$ is the trace of $M'$ in Fac $M$:
$$0\rightarrow tM'\overset{i}\rightarrow M'\overset{p}\rightarrow N\rightarrow 0$$
By the properties of torsion pairs we have Hom$_A(M,N)=0$. 
On the other hand, $\Hom_A(P,M')=\Hom_A(M',\tau M)=0$ because $(M\oplus M',P)$ is a $\tau$-tilting pair. 
The projectivity of $P$ implies that $\Hom_A(P,N)=0$. 
Moreover, any morphism $f:N\ra \tau M$ induces a morphism $fp:M'\ra \tau M$, hence $\Hom_A(N,\tau M)=0$.

Since $N$ belongs to $ M^{\perp} \cap {^\perp}(\tau M) \cap P^{\perp}$, we know that $N$ is an $\alpha(M,P)$-stable module. 
Hence the positive cone induced by the  almost $\tau$-tilting pair $(M,P)$ is included in the stability space $\D(N)$. 
Moreover, the positive cone has codimension 1, implying that $\D(N)$ is actually a wall.
\end{proof}

As a corollary, we get the following result, which completely describes  the chamber induced by a $\tau$-tilting pair $(M,P)$.

\begin{cor}\label{walls&chambers}
Let $(M,P)$ be a $\tau$-tilting pair. 
Then $(M,P)$ induces a chamber $\Ch_{(M,P)}$ having exactly $n$ walls $\{\D(N_1), \dots, \D(N_n)\}$. 
The stable modules $\{N_1, \dots, N_n\}$ can be obtained from the $n$ almost $\tau$-tilting pairs contained in $(M,P)$ as described in Proposition \ref{walls}.
\end{cor}

\begin{proof}
Suppose that $M=\bigoplus_{i=1}^k M_i$ and $P=\bigoplus_{j=k+1}^nP_j$. 
The fact that $(M,P)$ induces a chamber follows from Proposition~\ref{chambers}. 
Moreover, we know from Proposition \ref{walls} that each of the $n$ almost $\tau$-tilting pairs that are direct summand of $(M,P)$ are included in a wall.  
Finally, the fact that the $\D(N_i)$ are pairwise distinct follows from the fact that the set of $g$-vectors $\{g^{M_1}, \dots, g^{M_k}, -g^{P_{k+1}}, \dots, -g^{P_{t}}\}$ forms a basis, as shown in \cite[Theorem 5.1]{AIR}.
\end{proof}

\begin{rmk}
Not every wall is generated by the positive cone of some almost $\tau$-tilting pair. 
For instance, take the hereditary algebra of the Kronecker quiver $\xymatrix{1\ar@<0.5 ex>[r]\ar@<-0.5 ex>[r]&2}$ and consider the wall $\mathfrak{D}\left(\rep{1\\2}\right)=\{\lambda(1,-1):\lambda>0\}$. 
In this case there is no $\tau$-rigid module $M$ such that $g^M=(1,-1)$.
Note that the wall $\mathfrak{D}$ has no adjacent chamber, but rather is a limit of walls given by preprojective (or preinjective) modules.

Also the number of positive cones defined by almost $\tau$-tilting pairs included in a given wall is not constant. 
For instance $\D\left(\rep{1\\22}\right)=\left\{\lambda\left(g^{\rep{11\\222}}\right): \lambda\in\mathbb{R}_{\geq 0}\right\}$ while $\D(\rep{2})=\left\{\lambda\left(g^{\rep{1\\22}}\right): \lambda\in\mathbb{R}_{\geq 0}\right\}\cup\left\{\lambda\left(-g^{\rep{1\\22}}\right): \lambda\in\mathbb{R}_{\geq 0}\right\}$
\end{rmk}

The next example is intended to be an illustration of our previous propositions.

\begin{ex}
Consider the wall and chamber structure of the algebra $\mathbb{A}_2$ as we did in Example \ref{runningexample}. 
In Figure \ref{tauw&cA_2} we can see how positive cones of $\tau$-tilting pairs coincide with chambers and how positive cones of almost $\tau$-tilting pairs are included in walls of the wall and chamber structure of $\mathbb{A}_2$.

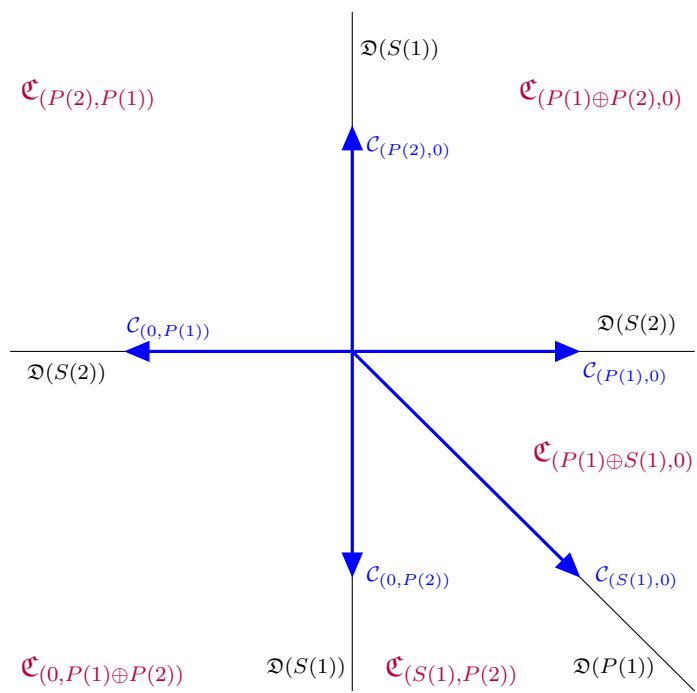
\begin{figure}
			\begin{center}
				\begin{tikzpicture}[line cap=round,line join=round,>=triangle 45,x=3.0cm,y=3.0cm]
					\clip(-1.5,-1.5) rectangle (1.5,1.5);
						\draw (0.,0.) -- (0.,1.5);
						\draw (0.,0.) -- (0.,-1.5);
						\draw [domain=-1.5:0.0] plot(\x,{(-0.-0.*\x)/-1.});
						\draw [domain=0.0:1.5] plot(\x,{(-0.-1.*\x)/1.});
						\draw [domain=0.0:1.5] plot(\x,{(-0.-0.*\x)/1.});
						\draw [->,line width=1.2pt,color=blue] (0.,0.) -- (1.,0.);
						\draw [->,line width=1.2pt,color=blue] (0.,0.) -- (0.,1.);
						\draw [->,line width=1.2pt,color=blue] (0.,0.) -- (-1.,0.);
						\draw [->,line width=1.2pt,color=blue] (0.,0.) -- (0.,-1.);
						\draw [->,line width=1.2pt,color=blue] (0.,0.) -- (1.,-1.);
						\draw[color=purple] (1.5,1.25) node[anchor=north east] {$\Ch_{(P(1)\oplus P(2),0)}$};
						\draw[color=purple] (-1.5,1.25) node[anchor=north west] {$\Ch_{(P(2),P(1))}$};
						\draw[color=purple] (1.55,-0.35) node[anchor=north east] {$\Ch_{(P(1)\oplus S(1), 0)}$};
						\draw[color=purple] (0.1,-1.3) node[anchor=north west] {$\Ch_{(S(1),P(2))}$};
						\draw[color=purple] (-1.5,-1.3) node[anchor=north west] {$\Ch_{(0,P(1)\oplus P(2))}$};
					\begin{scriptsize}
						\draw[color=black] (0,1.25) node[anchor= south west] {$\mathfrak{D}(S(1))$};
						\draw[color=black] (-0.2,-1.4) node {$\mathfrak{D}(S(1))$};
						\draw[color=black] (-1.25,-0.1) node {$\mathfrak{D}(S(2))$};
						\draw[color=black] (1.15,-1.4) node {$\mathfrak{D}(P(1))$};
						\draw[color=black] (1.25,0.12) node {$\mathfrak{D}(S(2))$};
						\draw[color=blue] (1.2,-0.1) node {$\Co_{(P(1),0)}$};
						\draw[color=blue] (0.25,0.9) node {$\Co_{(P(2),0)}$};
						\draw[color=blue] (-0.8,0.1) node {$\Co_{(0,P(1))}$};
						\draw[color=blue] (0.25,-1.0) node {$\Co_{(0,P(2))}$};
						\draw[color=blue] (1.25,-1.0) node {$\Co_{(S(1),0)}$};
					\end{scriptsize}
				\end{tikzpicture}
			\end{center}
			\caption{Wall and chamber structure for $\mathbb{A}_2$}
            \label{tauw&cA_2}
		\end{figure}
\end{ex}

 The previous example is a particular case of a more general phenomenon. Following the terminology introduced in \cite{DIJ}, we say that an algebra $A$ is {\em $\tau$-tilting finite} if there are only finitely many $\tau$-tilting pairs in $\mod A$. In \cite[Theorem 5.4]{DIJ}, Demonet, Iyama and Jasso show that the simplicial complex of a $\tau$-tilting finite algebra is homeomorphic to the $(n-1)$-dimensional sphere. Therefore we have the following corollary from the results of this subsection.

 \begin{cor}
 Let $A$ be a $\tau$-tilting finite algebra. Then the $g$-vectors of indecomposable $\tau$-rigid pairs determine completely the wall and chamber structure of $A$. 
 \end{cor}

We conjecture that the result of the corollary holds true for more general cases, maybe all finite-dimensional algebras.
It holds for hereditary algebras of tame type,  since the union of the chambers coming from silting objects has full measure, see \cite[Theorem 5.1(4)]{Hil}.
Moreover, the work of \cite{BHIT} might  extend the conjecture to cluster-tilted algebras of tame type. 

Note that the union of the chambers coming from silting objects has full measure precisely when the closure of the set of walls has zero measure.
 This leads to the following conjecture (which holds for tame quivers by Hille's result \cite{Hil}):

\begin{conj} If an algebra is tame, then the closure of the union of all walls has measure zero.
\end{conj}

Since some wild algebras may have the same property, this raises the question of what it means for an algebra to be $\tau-$tilting tame.
We suggest the following more general definition:

\begin{defi} We say that an algebra is $\tau$-tilting tame if the closure of the union of all walls has measure zero.
\end{defi}
It would be interesting to characterize these algebras algebraically.

\subsection{Torsion classes associated to chambers}\label{ssc:torsionchambers}

As we proved in the last subsection, every $\tau$-tilting pair $(M,P)$ defines a chamber $\Ch_{(M,P)}$ in the wall and chamber structure. 
In this subsection we associate to a given chamber $\Ch$ a torsion class $\T_{\Ch}$ and we show that $\T_{\Ch_{(M,P)}}=\Fac M$ if the chamber $\Ch_{(M,P)}$ is defined by a $\tau$-tilting pair $(M,P)$.

Recall that Bridgeland associated in \cite[Lemma 6.6]{B16} a torsion class $\T_{\theta}$ and a torsion free class $\F_\theta$ to every $\theta \in \R^n$ as follows.
$$\T_{\theta}=\{M\in \mod A:\langle \theta,N\rangle \geq 0 \text{ for every quotient $N$ of $M$}\}$$
$$\F_{\theta}=\{M\in \mod A:\langle \theta,L\rangle \leq 0 \text{ for every submodule $L$ of $M$}\}$$

As a direct consequence of these definitions is the following result.

\begin{prop}\label{wide}
Let $A$ be an algebra and $\theta \in \mathbb{R}^n$. 
Then the category of $\theta$-semistable modules is a wide subcategory of $\emph{mod} A$, \emph{i.e.}, is closed under kernels, cokernels and extensions.
\end{prop}

\begin{proof}
Let $\theta \in \mathbb{R}^n$. 
Then the linearity of the dot product implies that  
$$\T_\theta \cap \F_\theta = \{M \in \mod A : \langle \theta , M \rangle = 0 \text{ and } \langle \theta , L \rangle \leq 0 \text{ for every submodule $L$ of $M$}\}.$$ 
By definition \ref{stable}, we have that $\T_\theta \cap \F_\theta$ is the category of $\theta$-semistable modules. 
Hence the category of $\theta$-semistable modules is closed under extensions because both $\T_\theta$ and $\F_\theta$ are closed under extensions.

Now we prove that its is closed under kernels and cokernels.
Let $M_1$ and $M_2$ two $\theta$-semistable modules and let $f:M_1 \to M_2$ be a morphism of $A$-modules.
If $f$ is zero or an isomorphism, the result follows at once.  
Otherwise, consider the following short exact sequences in $\mod A$
$$ 0 \to \ker f \to M_1 \to \im f \to 0$$
$$ 0 \to \im f \to M_2 \to \coker f \to 0.$$
Given that $M_1$ is $\theta$-semistable we have that $\langle \theta , \im f\rangle \geq \langle \theta , \im f\rangle = 0$. 
Meanwhile, $\langle \theta , \im f \rangle \leq \langle \theta , M_2\rangle = 0$ because $M_2$ is $\theta$-semistable.
Consequently $\langle \theta , \im f\rangle =0$. 
The linearity of the function $\langle \theta , - \rangle: K_0(A) \otimes \mathbb{R}^n \to \mathbb{R}$ yields $\langle \theta , \ker f \rangle =0$ and $\langle \theta , \coker f \rangle=0$. 

Moreover, every submodule $L$ of $\ker f$ is a submodule of $M_1$, thus $\langle \theta , L \rangle \le \langle \theta , M_1 \rangle = 0$. 
Therefore $\ker f$ is $\theta$-semistable. 
A dual argument shows that $\coker f$ is also $\theta$-semistable.
\end{proof}

As a corollary of the previous result and lemma \ref{catsemistables} we have the following.

\begin{cor}
Let $(M,P)$ be a $\tau$-rigid pair. 
Then $M^\perp \cap {^\perp \tau M} \cap P^\perp$ is a wide subcategory of $\emph{mod} A$.
\end{cor}

\begin{proof}
Let $(M,P)$ be a $\tau$-rigid pair and let $\alpha(M,P)\in\rC_(M,P)$. 
Then lemma \ref{catsemistables} implies that $M^\perp \cap {^\perp \tau M} \cap P^\perp$ coincides with the category of $\alpha(M,P)$-semistable modules. 
Moreover, proposition \ref{wide} implies that $M^\perp \cap {^\perp \tau M} \cap P^\perp$ is a wide subcategory of $\mod A$.
\end{proof}

Now we use the notion of $\T_\theta$ to define a torsion class for a given chamber.

\begin{lem}
Let $\Ch$ be a chamber and consider the intersection 
$$\T_{\Ch}=\bigcap\limits_{\theta\in\Ch} \T_{\theta}$$
of all $\T_{\theta}$ when $\theta$ is in  $\Ch$. Then $\T_{\Ch}$ is a torsion class.
\end{lem}

\begin{proof}
This follows directly from \cite[Lemma 6.6]{B16} and the fact that an arbitrary intersection of torsion classes is again a torsion class.
\end{proof}

\begin{prop}\label{eqtorsionch}
Let $(M,P)$ be a $\tau$-tilting pair and $\Ch_{(M,P)}$ its induced chamber.
Then $\T_{\Ch_{(M,P)}}=\emph{Fac}M$.
\end{prop}

\begin{proof}
We first show that $\T_{\Ch_{(M,P)}} \subseteq \Fac M $. 
The torsion pair $(\Fac M, M^\perp)$ yields for any $N\in\mod A$ the canonical short exact sequence 
$$0\ra tN \ra N \ra N/tN \ra 0$$
with $tN \in \Fac M$. 
The dual of Lemma \ref{lemtrace} states that $\langle \theta, N/tN\rangle\leq 0$ for every $\theta\in\Ch_{(M,P)}$.
Hence $N\in\T_{\Ch_{(M,P)}}$ implies that $N/tN=0$, that is, $N\in\Fac M$.

Conversely, if $N\in\Fac M$, every quotient $N'$ of $N$ belongs to $\Fac M$ because $\Fac M$ is a torsion class.
Therefore Lemma \ref{lemtrace} gives that $\langle \theta, N'\rangle > 0$ for every $\theta\in\Ch_{(M,P)}$ and every $N'\neq0$.
Hence, $N\in \T_\theta$ for every $\theta \in \Ch_{(M,P)}$, and therefore 
$$N\in \bigcap_{\theta\in \Ch_{(M,P)}} \T_\theta=\T_{\Ch_{(M,P)}}.$$
\end{proof}

\begin{rmk}\label{sametorsion}
Note that, if a chamber $\Ch$ is induced by a $\tau$-tilting pair $(M,P)$, then $\T_{\theta}=\Fac M$ for every $\theta\in\Ch$ by Theorem \ref{stablemodcat}.
\end{rmk}

Following the terminology introduced in \cite{DIJ} we say that an algebra $A$ is \textit{$\tau$-tilting finite} if the  number of indecomposable $\tau$-rigid modules is finite (up to isomorphism). 

\begin{cor}
Let $A$ be an algebra. Then the function $\Ch$ mapping a $\tau$-tilting pair $(M,P)$ to its corresponding chamber $\Ch_{(M,P)}$ is injective. Moreover, if $A$ is $\tau$-tilting finite then $\Ch$ is also surjective.
\end{cor}
\begin{proof}
This follows directly from the fact that $\Ch_{(M,P)}=\text{Fac}M$ and \cite[Theorem 2.7]{AIR}. The moreover part follows from \cite[Theorem 1.5]{DIJ}
\end{proof}

\subsection{A detailed example}

During this section, we have illustrated each of our results using the hereditary algebra of type $\mathbb{A}_2$. 
The simplicity of the module category of this algebra is handy to provide simple (counter)examples. 
However, at the same time, its simplicity also can be misleading when trying to understand the general picture. 
We therefore finish this section giving a detailed example of a slightly bigger algebra. 
\bigskip

\begin{ex}\label{ex:Nakayama}
Let $A$ be the path algebra of the quiver 
$$
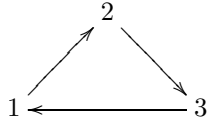
\xymatrix{
  & 2\ar[dr]& \\
  1\ar[ru] & & 3\ar[ll] }$$
quotiented by the third power of its Jacobson radical.
The Auslander-Reiten quiver of $A$ can be seen in figure \ref{fig:Ar-quiverA}.
\begin{figure}
    \centering
			\begin{tikzpicture}[line cap=round,line join=round ,x=2.0cm,y=1.8cm]
				\clip(-2.2,-0.1) rectangle (4.1,2.5);
					\draw [->] (-0.8,0.2) -- (-0.2,0.8);
					\draw [->] (1.2,0.2) -- (1.8,0.8);
					\draw [->] (3.2,0.2) -- (3.8,0.8);
					\draw [->] (-1.8,1.2) -- (-1.2,1.8);
					\draw [->] (0.2,1.2) -- (0.8,1.8);
					\draw [->] (2.2,1.2) -- (2.8,1.8);
					\draw [dashed] (-0.8,0.0) -- (0.8,0.0);
					\draw [dashed] (1.2,0.0) -- (2.8,0.0);
					\draw [dashed] (-1.8,1.0) -- (-0.2,1.0);
					\draw [dashed] (0.2,1.0) -- (1.8,1.0);
					\draw [dashed] (2.2,1.0) -- (3.8,1.0);
					\draw [dashed] (-2.0,0.0) -- (-1.2,0.0);
					\draw [dashed] (3.2,0.0) -- (4.0,0.0);
					\draw [->] (0.2,0.8) -- (0.8,0.2);
					\draw [->] (2.2,0.8) -- (2.8,0.2);
					\draw [->] (-1.8,0.8) -- (-1.2,0.2);
					\draw [->] (-0.8,1.8) -- (-0.2,1.2);
					\draw [->] (1.2,1.8) -- (1.8,1.2);
					\draw [->] (3.2,1.8) -- (3.8,1.2);
				
				\begin{scriptsize}
					\draw[color=black] (-1,0) node {$\rep{2}$};
					\draw[color=black] (1,0) node {$\rep{1}$};
					\draw[color=black] (3,0) node {$\rep{3}$};
					\draw[color=black] (-2,1) node {$\rep{2\\3}$};
					\draw[color=black] (0,1) node {$\rep{1\\2}$};
					\draw[color=black] (2,1) node {$\rep{3\\1}$};
					\draw[color=black] (4,1) node {$\rep{2\\3}$};
					\draw[color=black] (-1,2) node {$\rep{1\\2\\3}$};
					\draw[color=black] (1,2) node {$\rep{3\\1\\2}$};
					\draw[color=black] (3,2) node {$\rep{2\\3\\1}$};
				\end{scriptsize}
			\end{tikzpicture}
\caption{The Auslander-Reiten quiver of $A$}
    \label{fig:Ar-quiverA}
\end{figure}
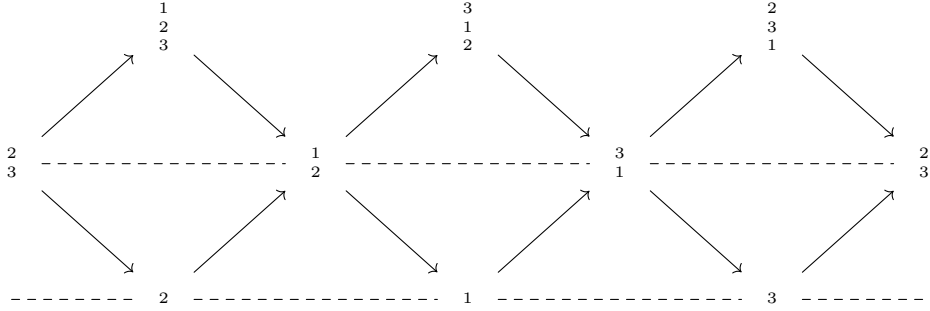
Note that every module is represented by its Loewy series and both copies of $\rep{2\\3}$ should be identified, so the Auslander-Reiten quiver of $A$ has the shape of a cylinder. 

Studying the $\tau$-tilting theory of $A$, is easy to see that there are 12 indecomposable $\tau$-rigid pairs in $\mod A$, since every indecomposable $A$-module is $\tau$-rigid. 
What is less obvious is that this algebra has 20 $\tau$-tilting pairs.
In table \ref{table:chambers} we give a complete list of the $\tau$-tilting pairs with their respective basis of $g$-vectors. 
Finally, in Figure \ref{fig:WCNakayama}, one can find the stereographic projection (from the point $(1,1,1)$) of the intersection of the unit sphere with the wall and chamber structure of $A$. 
The $g$-vectors of the indecomposable $\tau$-rigid pairs appear as vertices in Figure \ref{fig:WCNakayama}. They are visualized in three dimensions in \cite[Figure 1]{Yur}.
The chambers are labelled $\Ch_i$ where the subindex $i$ corresponds to the row number of the corresponding $\tau$-tilting pair in Table 1.
The stability space of each indecomposable module has a different color, while the $g$-vectors are written in red. 
Note that the stability spaces $\D\left(\rep{1\\2\\3}\right)$, $\D\left(\rep{2\\3\\1}\right)$ and $\D\left(\rep{3\\1\\2}\right)$ of $\rep{1\\2\\3}$, $\rep{2\\3\\1}$ and $\rep{3\\1\\2}$, respectively, fill up the plane orthogonal to the vector $(1,1,1)$.
This plane appears as the circle in Figure \ref{fig:WCNakayama} with three colors.
\begin{table}
\centering
\begin{tabular}{|l|c|c|}\hline
Chamber & $\tau$-tilting pair & $g$-vectors \\\hline
1 & $\left( \rep{1\\2\\3} \oplus \rep{2\\3\\1} \oplus \rep{3\\1\\2}, 0 \right)$ & $\left\{ \gvec{1\\0\\0}, \gvec{0\\1\\0}, \gvec{0\\0\\1} \right\}$ \\\hline
2 & $\left( \rep{1\\2\\3} \oplus \rep{2\\3\\1} \oplus \rep{2}, 0 \right)$ & $\left\{ \gvec{1\\0\\0}, \gvec{0\\1\\0}, \gvec{0\\1\\-1} \right\}$ \\\hline
3 & $\left( \rep{1\\2\\3} \oplus \rep{3\\1\\2} \oplus \rep{1}, 0 \right)$ & $\left\{ \gvec{1\\0\\0}, \gvec{0\\0\\1}, \gvec{1\\-1\\0} \right\}$\\\hline
4 & $\left( \rep{2\\3\\1} \oplus \rep{3\\1\\2} \oplus \rep{3}, 0 \right)$ & $\left\{ \gvec{0\\1\\0}, \gvec{0\\0\\1}, \gvec{-1\\0\\1} \right\}$ \\\hline
5 & $\left( \rep{1\\2\\3} \oplus \rep{1\\2} \oplus \rep{2}, 0 \right)$ & $\left\{ \gvec{1\\0\\0}, \gvec{1\\0\\-1}, \gvec{0\\1\\-1} \right\}$ \\\hline
6 & $\left( \rep{1\\2\\3} \oplus \rep{1\\2} \oplus \rep{1}, 0 \right)$ & $\left\{ \gvec{1\\0\\0}, \gvec{1\\0\\-1}, \gvec{1\\-1\\0} \right\}$ \\\hline
7 & $\left( \rep{2\\3\\1} \oplus \rep{2\\3} \oplus \rep{3}, 0 \right)$ & $\left\{ \gvec{0\\1\\0}, \gvec{-1\\1\\0}, \gvec{-1\\0\\1} \right\}$ \\\hline
8 & $\left( \rep{2\\3\\1} \oplus \rep{2\\3} \oplus \rep{2}, 0 \right)$ & $\left\{ \gvec{0\\1\\0}, \gvec{-1\\1\\0}, \gvec{0\\1\\-1} \right\}$ \\\hline
9 & $\left( \rep{3\\1\\2} \oplus \rep{3\\1} \oplus \rep{1}, 0 \right)$ & $\left\{ \gvec{0\\0\\1}, \gvec{0\\-1\\1}, \gvec{1\\-1\\0} \right\}$ \\\hline
10 & $\left( \rep{3\\1\\2} \oplus \rep{3\\1} \oplus \rep{3} , 0 \right)$ & $\left\{ \gvec{0\\0\\1}, \gvec{0\\-1\\1}, \gvec{-1\\0\\1} \right\}$\\\hline
11& $\left( \rep{3\\1} \oplus \rep{1}, \rep{2\\3\\1} \right)$ & $\left\{ \gvec{0\\-1\\1}, \gvec{1\\-1\\0}, \gvec{0\\-1\\0} \right\}$ \\\hline
12& $\left( \rep{3\\1} \oplus \rep{3}, \rep{2\\3\\1} \right)$ & $\left\{ \gvec{0\\-1\\1}, \gvec{-1\\0\\1}, \gvec{0\\-1\\0} \right\}$ \\\hline
13& $\left( \rep{2\\3} \oplus \rep{3}, \rep{1\\2\\3}\right)$ & $\left\{ \gvec{-1\\1\\0}, \gvec{-1\\0\\1}, \gvec{-1\\0\\0} \right\}$ \\\hline
14& $\left( \rep{2\\3} \oplus \rep{2}, \rep{1\\2\\3} \right)$ & $\left\{ \gvec{-1\\1\\0}, \gvec{0\\1\\-1}, \gvec{-1\\0\\0} \right\}$ \\\hline
15& $\left( \rep{1\\2} \oplus \rep{2}, \rep{3\\1\\2}\right)$ & $\left\{ \gvec{1\\0\\-1}, \gvec{0\\1\\-1}, \gvec{0\\0\\-1} \right\}$ \\\hline
16& $\left( \rep{1\\2} \oplus \rep{1}, \rep{3\\1\\2}\right)$ & $\left\{ \gvec{1\\0\\-1}, \gvec{1\\-1\\0}, \gvec{0\\0\\-1} \right\}$ \\\hline
17& $\left( \rep{3}, \rep{1\\2\\3} \oplus \rep{3\\1\\2} \right)$ & $\left\{ \gvec{-1\\0\\1}, \gvec{0\\-1\\0}, \gvec{-1\\0\\0} \right\}$ \\\hline
18& $\left( \rep{2}, \rep{1\\2\\3} \oplus \rep{3\\1\\2}\right)$ & $\left\{ \gvec{0\\1\\-1}, \gvec{-1\\0\\0}, \gvec{0\\0\\-1} \right\}$ \\\hline
19& $\left( \rep{1}, \rep{2\\3\\1} \oplus \rep{3\\1\\2}\right)$ & $\left\{ \gvec{1\\-1\\0}, \gvec{0\\-1\\0}, \gvec{0\\0\\-1} \right\}$ \\\hline
20 & $\left( 0, \rep{1\\2\\3} \oplus \rep{2\\3\\1} \oplus \rep{3\\1\\2}\right)$ & $\left\{ \gvec{-1\\0\\0}, \gvec{0\\-1\\0}, \gvec{0\\0\\-1} \right\}$  \\\hline
\end{tabular}
\caption{The list of $\tau$-tilting pairs in $\mod A$}
\label{table:chambers}
\end{table}

\definecolor{ffzztt}{rgb}{1,0.6,0.2}
\definecolor{qqccqq}{rgb}{0,0.8,0}
\definecolor{wwqqcc}{rgb}{0.4,0,0.8}
\definecolor{yqqqyq}{rgb}{0.5019607843137255,0,0.5019607843137255}
\definecolor{qqttqq}{rgb}{0,0.2,0}
\definecolor{qqttcc}{rgb}{0,0.2,0.8}
\definecolor{ffqqtt}{rgb}{1,0,0.2}
\definecolor{ffqqqq}{rgb}{1,0,0}
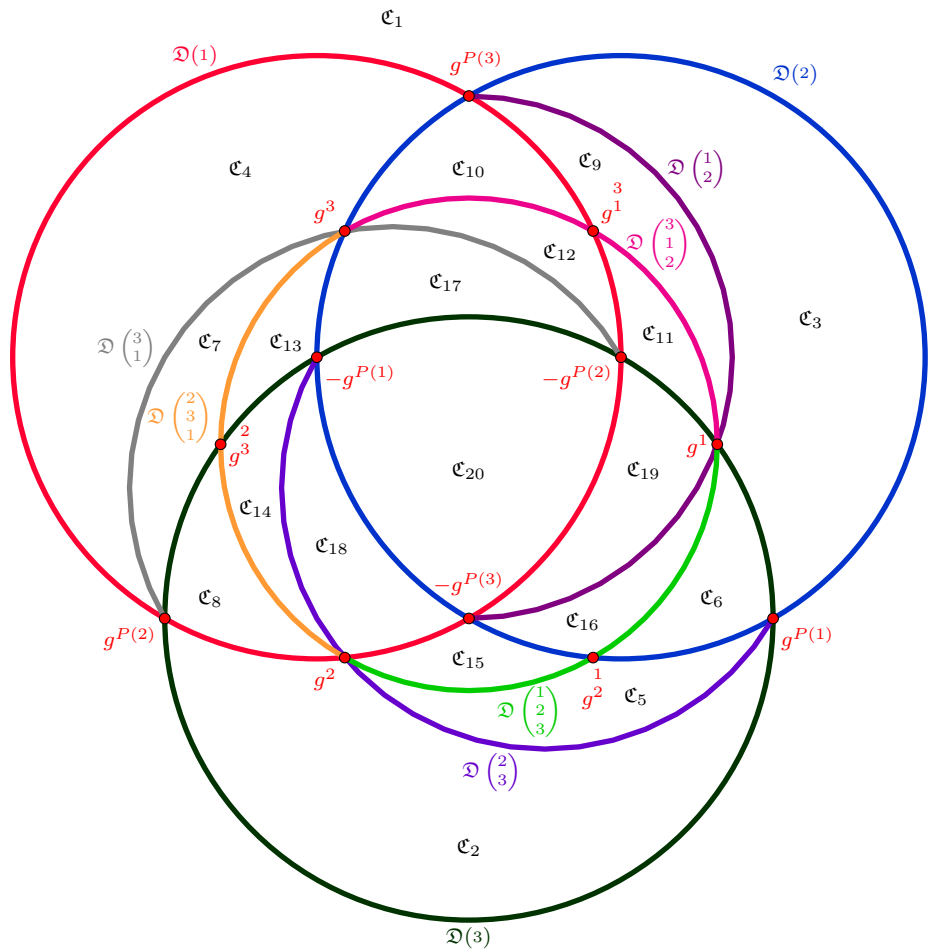
\begin{figure}
    \centering
    \begin{tikzpicture}[line cap=round,line join=round,>=triangle 45,x=1cm,y=1cm]
\clip(-7.843950349746072,-9.0) rectangle (8.613158409540794,6.839877367676253);
\draw [line width=2pt,color=ffqqtt] (-2,0) circle (4cm);
\draw [line width=2pt,color=qqttcc] (2,0) circle (4cm);
\draw [line width=2pt,color=qqttqq] (0,-3.464101615137755) circle (4cm);
\draw [shift={(0,0)},line width=2pt,color=yqqqyq]  plot[domain=-1.5707963267948966:1.5707963267948966,variable=\t]({1*3.4641016151377544*cos(\t r)+0*3.4641016151377544*sin(\t r)},{0*3.4641016151377544*cos(\t r)+1*3.4641016151377544*sin(\t r)});
\draw [shift={(1,-1.7320508075688772)},line width=2pt,color=wwqqcc]  plot[domain=2.6179938779914944:5.759586531581288,variable=\t]({1*3.4641016151377544*cos(\t r)+0*3.4641016151377544*sin(\t r)},{0*3.4641016151377544*cos(\t r)+1*3.4641016151377544*sin(\t r)});
\draw [shift={(-1,-1.7320508075688772)},line width=2pt,color=gray]  plot[domain=0.5235987755982988:3.665191429188092,variable=\t]({1*3.4641016151377544*cos(\t r)+0*3.4641016151377544*sin(\t r)},{0*3.4641016151377544*cos(\t r)+1*3.4641016151377544*sin(\t r)});
\draw [shift={(0,-1.154700538379251)},line width=2pt,color=magenta]  plot[domain=0:2.0943951023931957,variable=\t]({1*3.2659863237109037*cos(\t r)+0*3.2659863237109037*sin(\t r)},{0*3.2659863237109037*cos(\t r)+1*3.2659863237109037*sin(\t r)});
\draw [shift={(0,-1.1547005383792532)},line width=2pt,color=qqccqq]  plot[domain=-2.0943951023931957:0,variable=\t]({1*3.265986323710903*cos(\t r)+0*3.265986323710903*sin(\t r)},{0*3.265986323710903*cos(\t r)+1*3.265986323710903*sin(\t r)});
\draw [shift={(0,-1.154700538379251)},line width=2pt,color=ffzztt]  plot[domain=2.0943951023931953:4.188790204786391,variable=\t]({1*3.2659863237109037*cos(\t r)+0*3.2659863237109037*sin(\t r)},{0*3.2659863237109037*cos(\t r)+1*3.2659863237109037*sin(\t r)});
\begin{scriptsize}
\draw [fill=ffqqqq] (-2,0) circle (2.0pt) node[align=right, anchor= north west, color=red] {$-g^{P(1)}$};
\draw [fill=ffqqqq] (2,0) circle (2.0pt) node[align=left, left, color=red, anchor=north east ] {$-g^{P(2)}$};
\draw [fill=ffqqqq] (0,-3.464101615137755) circle (2pt) ;
\draw [fill=ffqqqq] (0,3.4641016151377544) circle (2pt) ;
\draw [fill=ffqqqq] (4,-3.4641016151377544) circle (2pt) node[align=center, above, color=red, anchor=north west] {$g^{P(1)}$};
\draw [fill=ffqqqq] (-4,-3.4641016151377544) circle (2pt) node[align=center, above, color=red, anchor=north east] {$g^{P(2)}$};
\draw [fill=ffqqqq] (-1.6329931618554523,1.6737265863669384) circle (2pt) node[align=center, above, color=red, anchor=south east] {$g^{\rep{3}}$};
\draw [fill=ffqqqq] (3.2659863237109037,-1.1547005383792517) circle (2pt) node[align=center, above, color=red, anchor=east] {$g^{\rep{1}}$};
\draw [fill=ffqqqq] (-1.6329931618554516,-3.9831276631254413) circle (2pt) node[align=right, above, color=red, anchor=north east] {$g^{\rep{2}}$};
\draw [fill=ffqqqq] (1.632993161855452,1.6737265863669386) circle (2pt) node[align=center, above, color=red, anchor=south west] {$g^{\rep{3\\1}}$};
\draw [fill=ffqqqq] (1.632993161855452,-3.983127663125442) circle (2pt) node[align=center, above, color=red, anchor=north ] {$g^{\rep{1\\2}}$};
\draw [fill=ffqqqq] (-3.2659863237109046,-1.154700538379252) circle (2pt) node[align=center, color=red, anchor=west] {$g^{\rep{2\\3}}$};
\draw[color=red] (0, -3.0) node {$-g^{P(3)}$};
\draw[color=red] (0.1, 3.9) node {$g^{P(3)}$};
            \draw[color=ffqqtt] (-4,3.75) node[anchor=south west] {$\mathfrak{D}(\rep{1})$};
			\draw[color=qqttcc] (4.3,3.75) node {$\mathfrak{D}(\rep{2}$)};
			\draw[color=qqttqq] (0,-7.7) node {$\mathfrak{D}(\rep{3})$};
			\draw[color=yqqqyq] (3.0,2.5) node {$\mathfrak{D}\left(\rep{1\\2}\right)$};
			\draw[color=gray] (-4.5,0.12) node {$\mathfrak{D}\left(\rep{3\\1}\right)$};
			\draw[color=wwqqcc] (0.3,-5.5) node {$\mathfrak{D}\left(\rep{2\\3}\right)$};
			\draw[color=qqccqq] (0.8,-4.7) node {$\mathfrak{D}\left(\rep{1\\2\\3}\right)$};
			\draw[color=ffzztt] (-3.8,-0.8) node {$\mathfrak{D}\left(\rep{2\\3\\1}\right)$};
			\draw[color=magenta] (2.5,1.5) node {$\mathfrak{D}\left(\rep{3\\1\\2}\right)$};
			
			\draw (-1,4.5) node {$\Ch_{1}$};
			\draw (0,-6.5) node {$\Ch_{2}$};
			\draw (4.5,0.5) node {$\Ch_{3}$};
			\draw (-3,2.5) node {$\Ch_{4}$};
			\draw (2.2,-4.5) node {$\Ch_{5}$};
			\draw (3.2,-3.2) node {$\Ch_{6}$};
			\draw (-3.4,0.2) node {$\Ch_{7}$};
			\draw (-3.4,-3.2) node {$\Ch_{8}$};
			\draw (1.6,2.6) node {$\Ch_{9}$};
			\draw (0,2.5) node {$\Ch_{10}$};
			\draw (2.5,0.3) node {$\Ch_{11}$};
			\draw (1.2,1.4) node {$\Ch_{12}$};
			\draw (-2.4,0.2) node {$\Ch_{13}$};
			\draw (-2.8,-2) node {$\Ch_{14}$};
			\draw (0,-4) node {$\Ch_{15}$};
			\draw (1.5,-3.5) node {$\Ch_{16}$};
			\draw (-0.3,1.0) node {$\Ch_{17}$};
			\draw (-1.8,-2.5) node {$\Ch_{18}$};
			\draw (2.3,-1.5) node {$\Ch_{19}$};
			\draw (0,-1.5) node {$\Ch_{20}$};
\end{scriptsize}
\end{tikzpicture}
\caption{The stereographic projection of the wall and chamber structure of $A$}
\label{fig:WCNakayama}
\end{figure}
\end{ex}

\newpage
\section{$\D$-generic paths in the wall and chamber structure of an algebra}\label{Sc:MGSmodcat}

In this section we study paths in the wall and chamber structure of an algebra. Using the results of the previous section, we show that every maximal green sequence is induced by a path the wall and chamber structure. As an application we show that certain algebras do not admit maximal green sequences in their module category. 

\subsection{Mutations and $\mathfrak{D}$-generic paths}

The wall and chamber structure of an algebra is part of the scattering diagram of the algebra. 
Even if there are multiple scattering diagrams for a given algebra, in general one is interested only in those which are \textit{consistent}. 
The precise definition of consistency of scattering diagrams is beyond the scope this paper, but it relies heavily on a particular class of paths. 

There are multiple versions of paths defined in the literature, such as the broken lines in \cite{GHKK} or the straight lines considered by Igusa in \cite{Igu}.
We use here the $\mathfrak{D}$-generic paths defined by Bridgeland in \cite{B16} as follows.

\begin{defi}\cite[\S 2.7]{B16}\label{defDgeneric}
We say that a smooth path $\gamma:[0,1]\to \mathbb{R}^n$ is a $\mathfrak{D}$-generic path if:
\begin{enumerate}
\item $\gamma(0)$ and $\gamma(1)$ do not belong to the stability space $\D(M)$ of a nonzero $A$-module $M$, that is, they are located inside some chambers;
\item If $\gamma(t)$ belongs to the intersection $\D(M)\cap\D(N)$ of two walls, then the dimension vector $[M]$ of $M$ is a scalar multiple of the dimension vector $[N]$ of $N$;
\item whenever $\gamma(t)$ is in $\D(M)$, then $\langle \gamma'(t), [M] \rangle \neq 0$.
\end{enumerate}
\end{defi}

In other words, a smooth path is $\D$-generic if crosses one wall at a time and the crossing is transversal. 

Note that in Definition \ref{defDgeneric}.3 we only ask that $\langle \gamma'(t), [M] \rangle \neq 0$ for each $t$ such that $\gamma(t)\in\D(M)$ for some nonzero $A$-module $M$. 
We say more precisely that a crossing is \textit{green} if $\langle \gamma'(t), [M] \rangle > 0$. 
Otherwise we say that the crossing is \textit{red}. 
A $\D$-generic path is called {\em green} if all its crossings are green.

Consider now a set $\{(M_0, P_0), (M_1, P_1), \dots, (M_r, P_r)\}$ of $\tau$-tilting pairs such that $(M_i, P_i)$ is a mutation of $(M_{i-1}, P_{i-1})$. 
By Corollary \ref{walls&chambers}, each $(M_i, P_i)$ defines a chamber $\Ch_{(M_{i},P_i)}$ in the wall and chamber structure for $A$. 
Moreover, Proposition \ref{walls} ensures that consecutive chambers share a wall. 
This allows to construct a piecewise linear path $\gamma:[0,1] \to \mathbb{R}^n$ going in straight lines from the central element $g^{M_i}-g^{P_i}$ of the cone $\Ch_{(M_{i},P_i)}$ to the next one. 
The precise definition in parametric form is as follows:
$$\gamma(t)= (1-rt+i)(g^{M_i}-g^{P_i})-(rt-i)(g^{M_{i+1}}-g^{P_{i+1}}) \text{ if } t\in\left[ \frac{i}{r}, \frac{i+1}{r} \right].$$

By construction,  $\gamma\left(\frac{i}{r}\right) = g^{M_i}-g^{P_i} \in \Ch_{(M_i, P_i)}$ for every $0\leq i\leq r$. 
In Theorem \ref{Dgenericpath}, one of the main results of this paper, we show that every finite sequence of mutations is represented by a $\D$-generic path which is closely related to the path $\gamma$ that we just constructed. 
We first consider the case of one mutation:

\begin{lem}\label{lemDgeneric}
Let $(M_0, P_0)$ and $(M_1, P_1)$ be two $\tau$-tilting pairs for the algebra $A$ such that one is a mutation of the other. 
Then the path $\gamma:[0,1]\to \mathbb{R}^n$ defined by
$$\gamma(t)=(1-t)(g^{M_0}-g^{P_0})+t(g^{M_1}-g^{P_1}) \mbox{ for } t\in [0,1]$$
is a $\D$-generic path in the wall and chamber structure of $A$ that starts in $\Ch_{(M_0, P_0)}$, finishes in $\Ch_{(M_1, P_1)}$ and crosses only one wall. 
Moreover the crossing is green if and only if $M_1$ is a left mutation of $M_0$, that is, $\emph{Fac} M_0\subset \emph{Fac} M_1$.
\end{lem}

\begin{proof}
Of course, the point $\gamma(0)=g^{M_0}-g^{P_0}$ is given by the coordinate vector $\alpha(M_0, P_0) = (1, \ldots, 1)$ in the cone $\Co_{(M_0, P_0)}$ and thus belongs to the interior, likewise for $\gamma(1)=g^{M_1}-g^{P_1}$.  
Thus the property (1) in Definition \ref{defDgeneric} follows from Proposition \ref{chambers}.

Because $(M_1, P_1)$ is a mutation of $(M_0, P_0)$, there is an almost $\tau$-tilting pair $(M,P)$ such that $M$ is a direct factor of $M_0$ and $M_1$ while $P$ is a direct summand of $P_0$ and $P_1$. 
We have that $\Ch_{(M_0, P_0)}$ and $\Ch_{(M_1, P_1)}$ are neighboring chambers separated only by $\C_{(M,P)}$, which is contained in the a wall $\D(N)$ where $N$ is constructed in proposition \ref{walls}.
Moreover Theorem \ref{stablemodcat} implies the existence of a unique $\alpha(M,P)$-stable module $S$.
Hence $N$ and every other module $N'$ in the abelian category of $\alpha(M,P)$-semistable module has a dimension vector which is a scalar multiple of $S$. 
Therefore, if $\gamma$ crosses another wall $\D(N')$, then $N'$ is an $\alpha(M,P)$-semistable module, hence $[N']$ is a scalar multiple of $[N]$. 
This shows condition (2) in Definition \ref{defDgeneric}.

As the almost $\tau$-tilting pair $(M,P)$ is a direct factor of $(M_0, P_0)$ and $(M_1, P_1)$, we can write the $g$-vectors $g^{M_0}-g^{P_0}$ and $g^{M_1}-g^{P_1}$ of the $\tau$-tilting pairs $(M_0, P_0)$ and $(M_1, P_1)$  as 
\begin{equation}\label{eq1}
g^{M_0}-g^{P_0}=(g^M-g^P)+g'
\end{equation}
and 
\begin{equation}\label{eq2}
g^{M_1}-g^{P_1}=(g^M-g^P)+g'',
\end{equation} where $g'$ and $g''$ are the $g$-vectors of the complements of $(M,P)$ in $(M_0, P_0)$ and $(M_1, P_1)$, respectively.
This yields the following reformulation of the function $\gamma$:
\begin{equation}
\gamma(t)=(1-t)g'+t g''+(g^M-g^P)
\end{equation}
We therefore have $\gamma'(t)=-g'+g''$ for every $t\in[0,1]$. It remains to show that 
$$\langle\gamma'(t),[N]\rangle =\langle -g'+g'' , [N] \rangle\neq 0$$ where $N$ is constructed in proposition \ref{walls}.

Suppose that $\Fac M_0 \subset \Fac M_1$, then we know that $\Fac M=\Fac M_0$. Thus, we have that $(\T_0, \F_0)=(\Fac M, M^\perp)$ and $(T_1, \F_1)=(^\perp(\tau M)\cap P^\perp, \F_1)$ are the torsion pairs associated to $(M_0, P_0)$ and $(M_1, P_1)$, respectively. 
Moreover, we have that $N\in M^{\perp}\cap {^{\perp}(\tau M)}\cap P^\perp$ which is contained in $M^\perp = \F_0$ and, at the same time, is contained in $^{\perp}(\tau M)\cap P^\perp=\T_1$. Hence, Lemma \ref{lemtrace} implies the following. 
\begin{eqnarray}
0>\langle g^{M_0}-g^{P_0}, [N]\rangle=\langle g', [N]\rangle+\langle g^M-g^P, [N] \rangle=\langle g', [N]\rangle\\
0<\langle g^{M_1}-g^{P_1}, [N]\rangle=\langle g'', [N]\rangle+\langle g^M-g^P, [N] \rangle=\langle g'', [N]\rangle
\end{eqnarray}
Therefore we have that $$\langle\gamma'(t),[N]\rangle=\langle -g'+g'' , [N] \rangle=\langle -g', [N] \rangle + \langle g'', [N] \rangle > 0,$$ 
which means that the crossing is green. Using the same arguments one shows that if $\Fac M=\Fac M_1$, then the crossing is red. This completes the proof.
\end{proof}

Now we are able to prove the main result of this subsection. 

\begin{theorem}\label{Dgenericpath}
Let $\{(M_0, P_0), (M_1, P_1), \dots, (M_r, P_r)\}$ be a set of $\tau$-tilting pairs such that for every $i \ge 1$, $(M_i, P_i)$ is a mutation of $(M_{i-1}, P_{i-1})$. 
Then there exists a $\D$-generic path $\tilde{\gamma}:[0,1]\to \mathbb{R}^n$ with exactly $r$ points $\{t_1, \dots, t_r\}$ in $[0,1]$ where $\tilde{\gamma}(t_i)$ belongs to $\D(N_i)$ for some nonzero module $N_i$ and $\tilde{\gamma}(t)$ belong to $\Ch_{(M_i, P_i)}$ for every $t$ in the interval $(t_i, t_{i+1})$.
\end{theorem}

\begin{proof}
Consider the path $\gamma:[0,1]\to \mathbb{R}^n$ defined as before by
$$\gamma(t)= (1-rt+i)(g^{M_i}-g^{P_i})-(rt-i)(g^{M_{i+1}}-g^{P_{i+1}}) \text{ if } t\in\left[ \frac{i}{r}, \frac{i+1}{r} \right].$$
 
Of course, the function $\gamma$ restricted to the interval $[i/r, (i+1)/r]$  coincides with the path considered in Lemma \ref{lemDgeneric}. 
Therefore $\gamma$ crosses exactly $r$ walls verifying Definition \ref{defDgeneric}.1 \ref{defDgeneric}.2 and \ref{defDgeneric}.3. 
But $\gamma$ is not smooth and thus is not a $\D$-generic path. 

Note however that  every chamber $\Ch_{(M_i, P_i)}$ is an open set in $\mathbb{R}^n$ with the euclidean topology. 
Therefore, for every $1\leq i\leq r-1$ there is an $\epsilon_i>0$ such that the sphere $B(g^{M_i}-g^{P_i}, \epsilon_i)$ of center $g^{M_i}-g^{P_i}$ and radius $\epsilon_i$ is contained in $\Ch_{(M_i, P_i)}$. 
Choose a smooth path $\tilde{\gamma}:[0,1]\to \mathbb{R}^n$ such that $\gamma(t)=\tilde{\gamma}(t)$ when $\gamma(t)$ does not belong to $B(g^{M_i}-g^{P_i}, \epsilon_i)$ for every $i$. Then $\tilde{\gamma}$ crosses a wall exactly in the same points as $\gamma$. 
Therefore $\tilde{\gamma}$ is a $\D$-generic path crossing exactly $r$ walls. This finishes the proof.
\end{proof}

Given an algebra $A$ we can always construct the following graph associated to the wall and chamber structure of $A$.

\begin{defi}

Let $A$ be an algebra. We define the quiver $\mathfrak{G}_A$ as follows. 

	\begin{itemize}
		\item The vertices of $\mathfrak{G}_A$ correspond to the chambers in the wall and chamber structure of $A$.
		\item There is an arrow from the vertex associated to $\Ch_1$ to the vertex associated to $\Ch_2$ if $\T_{\Ch_2}\subsetneq\T_{\Ch_1}$ and there is no torsion class $\T$ such that $\T_{\Ch_2}\subsetneq\T\subsetneq\T_{\Ch_1}$.
	\end{itemize}
\end{defi}

As a immediate consequence of Theorem \ref{Dgenericpath} we obtain the following.

\begin{prop}
Let $A$ be an algebra. Then the exchange graph of $\tau$-tilting pairs of $A$ is a full subquiver of $\mathfrak{G}_A$. Moreover both quivers are isomorphic if $A$ is $\tau$-tilting finite. 
\end{prop}

\begin{rmk}
Suppose that we have two vertices of $\mathfrak{G}_A$ induced by $\tau$-tilting pairs $(M_1,P_1)$ and $(M_2, P_2)$ that are connected by an edge. 
Then it is easy to see that this edge corresponds to the cone $\C_{(M,P)}$ generated by the almost $\tau$-tilting pair $(M,P)$ which is a direct summand of both $(M_1,P_1)$ and $(M_2, P_2)$. 
But theorem \ref{stablemodcat} implies the existence of a unique $\alpha(M,P)$-stable module $B_{(M,P)}$ for every $\alpha(M,P)\in \C_{(M,P)}$.
Moreover $B_{(M,P)}$ is a brick by \cite[Theorem 1]{Ru} and independent of the choice of $\alpha(M,P)\in \C_{(M,P)}$ by construction.
Therefore, the previous proposition implies that $\mathfrak{G}_A$ induces a brick labeling of the edges in the exchange graph of $\tau$-tilting pairs of $\mod A$.

This labeling by bricks appeared independently in \cite{Asai,DIRRT,BCZ}.
Moreover, based on  the results developed here, the brick labeling was studied further in \cite{Tc-v}, showing that dimension vectors of these bricks correspond to the $c$-vectors of $\mod A$.
\end{rmk} 

\begin{conj}
The quiver $\mathfrak{G}_A$ is isomorphic to the exchange graph of $\tau$-tilting pairs for every algebra $A$.
\end{conj}

\subsection{Maximal green sequences as $\D$-generic paths}

We study in this section maximal green sequences in module categories: 

\begin{defi}
A \textit{maximal green sequence} in  $\rm{mod} A$ is a finite sequence of torsion classes $0=\T_0\subsetneq \T_1 \subsetneq \dots \subsetneq \T_{n-1}\subsetneq \T_r=\rm{mod} A$ such that for all $i\in\{1, 2, \dots, r\}$, the existence of a torsion class $\T$ satisfying $\mathcal{T}_i\subseteq\mathcal{T}\subseteq\mathcal{T}_{i+1}$ implies $\mathcal{T}=\mathcal{T}_i$ or $\mathcal{T}=\mathcal{T}_{i+1}$.
\end{defi}

As a first result, we provide a characterization of maximal green sequences in terms of $\tau$-tilting pairs.

\begin{prop}\label{SVMmod}
Let $A$ be an algebra and $$\{0\}=\T_{0}\subsetneq \T_{1} \subsetneq \cdots \subsetneq \T_{r}=\emph{mod} A$$ be a maximal green sequence in $\emph{mod} A$. Then there exists a set $\{(M_i,P_i)\}_{i=0}^r$ of $\tau$-tilting pairs such that $\emph{Fac} M_i=\T_i$ for all $1\leq i \leq r$.
\end{prop}

\begin{proof}
We construct the $\tau$-tilting pairs $(M_i,P_i)$ such that $\rm{Fac} M_i=\T_i$  inductively,  starting at $\T_0=\{0\}$. 
Of course, setting $(M_0,P_0)= (0,A)$ provides a $\tau$-tilting pair with $\Fac M_0=\{0\}.$ 

Suppose now that we already constructed a $\tau$-tilting pair $(M_i,P_i)$ such that $\T_i=\Fac M_i$. By definition of a maximal green sequence, we know that there is no torsion class strictly between $\T_{i}=\Fac M_{i}$ and $\T_{i+1}$. 
Therefore \cite[Theorem 3.1]{DIJ} implies the existence of a $\tau$-tilting pair $(M_{i+1},P_{i+1})$ such that $\T_{i+i}=\Fac M_{i+1}$. 

Finally this process will eventually stop given that maximal green sequences consist only of finitely many torsion classes. 
This finishes the proof. 
\end{proof}

As a consequence of Theorem \ref{Dgenericpath} and Proposition \ref{SVMmod} we can give a characterization of maximal green sequences in module categories in terms of $\D$-generic paths, which is one of the aims of this work. 
Remember that one can associate a torsion class $\T_{\theta}$ for every  $\theta \in \mathbb{R}^n$, see section \ref{ssc:torsionchambers}.

\begin{theorem}\label{MGSasDpaths}
Let $A$ be an algebra and $$\{0\}=\T_{0}\subsetneq \T_{1} \subsetneq \cdots \subsetneq \T_{r}=\mod A$$ be a maximal green sequence in $\mod A$. 
Then there exist a $\D$-generic path $\gamma:[0,1] \to \mathbb{R}^n$ such that the following conditions hold:
\begin{enumerate}
\item $\gamma$ crosses exactly $r$ walls at $t_1<t_2<\dots<t_r$;
\item every wall crossing of $\gamma$ is green;
\item the torsion class associated to $\T_{\gamma(t)}$ coincides with some torsion class $\T_k$ in the maximal green sequence for every $t\in[0,1]$;
\item $\T_{\gamma(t')}$ is contained in $\T_{\gamma(t'')}$ if $t'<t''$.
\end{enumerate}
\end{theorem}

\begin{proof}
Proposition \ref{SVMmod} yields a set of $\tau$-tilting pairs $\{(M_0, P_0), (M_1, P_1), \dots, (M_r, P_r)\}$ such that $(M_i, P_i)$ is a left mutation of $(M_{i-1}, P_{i-1})$ and $\T_i=\Fac M_i$. 
The existence of a $\D$-generic path $\gamma$ satisfying condition (1) follows directly from Theorem \ref{Dgenericpath}.
Moreover, since $\T_{i-1}\subset \T_{i}$, Lemma \ref{lemDgeneric} implies that every crossing is green, which shows condition (2). 

Suppose that $t$ is not a point in which $\gamma$ crosses a wall. 
Hence, Theorem \ref{Dgenericpath} shows that $\gamma(t)$ belongs to the chamber $\Ch_{(M_i, P_i)}$ for some $i$. 
Therefore Proposition \ref{eqtorsionch} implies that $\T_{\gamma(t)}=\Fac M_i$. 

Otherwise, if $t=t_i$ is a point in which $\gamma$ crosses a wall we have that $\gamma(t_i)\in \C_{(M,P)}$, where $(M,P)$ is the only almost $\tau$-tilting pair that is a direct summand of $(M_{i-1},P_{i-1})$ and $(M_i, P_i)$. 
Denote by $(\T_{i-1},\F_{i-1})$ and $(\T_i, \F_i)$ the torsion pairs associated to $(M_{i-1}, P_{i-1})$ and $(M_i, P_i)$, respectively. 
Hence Lemma \ref{lemtrace} implies that $\T_{i-1}\subsetneq \T_{\gamma(t)}$. Also, the dual of Lemma \ref{lemtrace} implies that $\F_{i}\subsetneq \F_{\gamma(t)}$, where $\F_{\gamma(t)}$ is the torsion free class associated to $\gamma(t)$. 
From these two equations we deduce that 
$$\Fac M_{i-1} \subsetneq \T_{\gamma(t)} \subset \Fac M_i,$$
which implies that $\T_{\gamma(t)} =\Fac M_i$ by the definition of a maximal green sequence. 
This shows condition (3). 
Finally, (4) is a direct consequence of (3).
\end{proof}

\begin{rmk}
Note that for the green $\D$-generic path constructed in corollary \ref{MGSasDpaths} we have $\gamma(0)=-g^A\in \Ch_{(0,A)}$ and
$\gamma(1)=g^A\in \Ch_{(A,0)}$ with $\T_{\gamma(1)}=\mod A$. Lemma \ref{lemDgeneric} implies that if we extend $\gamma$ in order to cross another wall, then this crossing must be red. Hence, the $\D$-generic paths associated to a maximal green sequence are the ones passing from $-g^A$ to
$g^A$ via a finite number of green crossings and which cannot be extended.
\end{rmk}

We finish the section with theorem \ref{Markoff} which shows an example that not every algebra admits a maximal green sequence.
These algebras are related to the cluster algebra of the one-punctured torus, and  have been object of intense studies in the context of cluster algebras, see for instance \cite[Example 35]{L-F-surfaces&potentials}, \cite{N-C-c&gvect} or \cite[Theorem 5.17]{DIJ}.
Before stating the theorem, we need some preparatory results.

\begin{lem}\label{firstchamber}
Let $A$ be an algebra and consider the two trivial $\tau$-tilting pairs $(A,0)$ and $(0,A)$. 
Then every wall in the boundary of $\Ch_{(A,0)}$ or $\Ch_{(0,A)}$ is defined by a simple $A$-module. 
\end{lem}

\begin{proof}
We show the result only for the $\tau$-tilting pair $(A,0)$, the case $(0,A)$ being analogous. 
For each $1 \le j \le n$, consider the almost $\tau$-tilting pair $\left( \bigoplus_{i\neq j} P(i), 0 \right)$ and the simple module $S(j)=P(j) / \top P(j)$. 
If $i\neq j$, we have that $$\langle g^{P(i)}, [S(j)]  \rangle= \langle e_i, e_j \rangle=0,$$
where $e_i$ and $e_j$ represent the $i$-th and $j$-th element of the canonical basis of $\mathbb{Z}^n$. 
Since $S(j)$ does not have any proper submodule, this implies that $S(j)$ is $\theta$-stable for all $\theta\in\C_{\left( \bigoplus_{i\neq j} P(i), 0 \right)}$ and the statement follows from the results in section \ref{chamber-results}.
\end{proof}

\begin{lem}\label{liftingwalls}
Let $A$ be an algebra and $I$ an ideal containing no non-zero idempotent of $A$. 
If $N$ is an $A/I$-module defining a wall $\D(N)$ in the wall and chamber structure of $A/I$ then $\D(N)$ is also a wall in the wall and chamber structure of $A$.
\end{lem}

\begin{proof}
Let $N$ be an $A/I$-module defining a wall $\D_{A/I}(N)$ and let $\theta\in\D_{A/I}(N)$. 
The submodules of $N$ as a $A$-module coincide with the submodules of $N$ as an $A/I$-submodule,
therefore $\theta(L)\leq 0$ for every submodule $L$ of $N$ as an $A$-module. This implies that $\theta\in\D_A(N)$. 
Hence $\D_A(N)$ is a wall in the wall and chamber structure of $A$.
The condition that the ideal $I$ contains no non-zero idempotent of $A$ is only there to guarantee that the Grothendieck groups of $A$ and $A/I$ are isomorphic, allowing us to compute $\theta(L)$ for both cases.
\end{proof}

Our last result is a generalization of \cite[Theorem 2.3.1]{Mmutnoninv}.

\begin{theorem}\label{Markoff}
 Let $A=kQ/I$ be an algebra where  $I$ is an admissible ideal of $kQ$ and the quiver $Q$ has exactly three vertices and admits the quiver
 $$
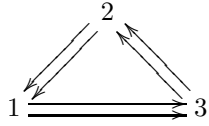
\xymatrix{
  & 2\ar@<0.5 ex>[dl]\ar@<-0.5 ex>[dl] & \\
  1\ar@<0.5 ex>[rr]\ar@<-0.5 ex>[rr] & & 3\ar@<0.5 ex>[ul]\ar@<-0.5 ex>[ul] }$$
 as a subquiver. 
 Then there is no maximal green sequence in $\mod A$. 
 \end{theorem}

\begin{proof}
We start the proof showing that if $C$ is the path algebra of the quiver 
$$
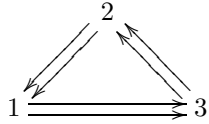
\xymatrix{
  & 2\ar@<0.5 ex>[dl]\ar@<-0.5 ex>[dl] & \\
  1\ar@<0.5 ex>[rr]\ar@<-0.5 ex>[rr] & & 3\ar@<0.5 ex>[ul]\ar@<-0.5 ex>[ul] }$$
 modulo its radical squared, then there is no maximal green sequence in $\mod C$. 
We do that by showing that every green $\D$-generic path $\gamma$ that starts at $(-1, -1, -1)$ and finishes in $(1,1,1)$ crosses infinitely many walls. 

Since there are three pairs of double arrows, we get from the representation theory of the Kronecker quiver the existence of three families of indecomposable modules in $\mod C$ whose dimension vectors are  as follows.
$$F_1=\{M_n\in \mod A : [M_n]=(n+1, 0, n), n \in \mathbb{N}\}$$
$$F_2=\{M'_n\in \mod A : [M_n]=(n, n+1, 0), n \in \mathbb{N}\}$$
$$F_3=\{M''_n\in \mod A : [M_n]=(0, n, n+1), n \in \mathbb{N}\}$$
For $M_n\in F_1$, 
we want to calculate its wall $\D(M_n)$.
The representation theory of the Kronecker quiver yields that the proper submodules of $M_n$ are the modules $M_k\in F_1$ with $k < n$. Therefore we have that 
$$\D(M_n)=\{(x, y, z) \in \mathbb{R}^3 : \langle(x,y,z), (n+1, 0, n)\rangle=0 \text{ and } \langle(x,y,z), (k+1, 0, k)\rangle\leq 0 \text{ for all $k<n$}\}$$
Therefore, $(x, y, z)\in \D(M_n)$ implies that 
\begin{equation}\label{eq1}
z=-\left(\frac{n+1}{n}\right) x
\end{equation}
and 
\begin{equation}\label{eq2}
0 \geq (k+1)x + kz \quad \text{ for all $k<n$.}
\end{equation}
Substituting \ref{eq1} in \ref{eq2} gives 
\begin{eqnarray}
0 \geq& (k+1)x -k\left(\frac{n+1}{n}\right) x \\
     =& \left(\frac{n(k+1) - k(n+1)}{n}\right) x \\
     =& \left(\frac{n-k}{n} \right) x
\end{eqnarray}
which implies that $x\leq 0$ because $k < n$. 
Therefore we have that 
$$\D(M_n)=\{(x, y, z) \in \mathbb{R}^3 : (n+1)x + nz=0 \text{ and } x\leq 0\}.$$
Likewise, one can prove that 
$$\D(M'_n)=\{(x, y, z) \in \mathbb{R}^3 : nx + (n+1)y=0 \text{ and } y\leq 0\},$$
$$\D(M''_n)=\{(x, y, z) \in \mathbb{R}^3 : ny + (n+1)z=0 \text{ and } z\leq 0\}$$

Now, suppose that we have a green $\D$-generic path $\gamma: [0,1] \to \mathbb{R}^3$ starting in $\gamma(0)=(-1,-1,-1)\in\Ch_{(0,C)}$ and ending in $\gamma(1)=(1,1,1)\in\Ch_{(C,0)}$.
We show that $\gamma$ crosses infinitely many walls.
Lemma \ref{firstchamber} implies that the first wall crossed by $\gamma$ is either $\D(S(1))$, $\D(S(2))$ or $\D(S(3))$. 
We consider the case were $\gamma$ starts crossing $\D(S(1))$, the other two cases are analogous. 

Because $\gamma$ crosses $\D(S(1))$ there exists a $t_0\in[0,1]$ such that $\gamma(t_0)=(0, y, z)$ with $y<0$ and $z<0$. 
Hence there exists an $\epsilon\in \mathbb{R}$ such that $\gamma(t_0+\epsilon)=(x, y, z)$ with $|x|<|y|$ and $y<0$, which gives  $nx + (n+1)y < 0$. Therefore, the wall $\D(M'_n)$ lies  between $\gamma(t_0+\epsilon)$ and $\gamma(1)=(1,1,1)$ for every $M'_n\in F_2$. 
Since  $\gamma$ is a green path, 
it must stay in the half-space $\{(x,y,z)\in \mathbb{R}^3 : x > 0\}$, implying that $\gamma$ crosses $\D(M'_n)$ for every $M'_n \in F_2$. 

This shows that every green $\D$-generic path $\gamma$ in the wall and chamber structure of $C$ starting in $\Ch_{(0,C)}$ and finishing in $\Ch_{C,0}$, crosses infinitely many walls. 

Now consider a finite dimensional algebra $A=kQ/I$ as in the statement of the theorem. 
Then lemma \ref{liftingwalls} implies that every $\D$-generic path $\gamma$ in the wall and chamber structure of $A$ crosses at least as many walls as the same path does in the wall and chamber structure of the algebra $C$. Therefore, every green $\D$-generic path $\gamma$ starting in $\Ch_{(0,A)}$ and $\Ch_{(A,0)}$ crosses infinitely many walls.
Henceforth, corollary \ref{MGSasDpaths} implies that there is no maximal green sequence in $\mod A$.
\end{proof}

\bibliography{BibliografiaTreffinger}{}

\bibliographystyle{abbrv}

\end{document}